\documentclass[10pt]{amsart}

\usepackage{times,amsmath,amsbsy,amssymb,amscd,mathrsfs}
\usepackage{slashbox}
\usepackage{graphicx,subfigure,epstopdf,wrapfig,chemarrow}
\usepackage{algorithm2e} 
\usepackage{multicol,multirow}
\usepackage{mathtools}
\usepackage[usenames,dvipsnames,svgnames,table]{xcolor}
\usepackage[numbered]{mcode}
\definecolor{myBlue}{rgb}{0.0,0.0,0.55}
\definecolor{green}{rgb}{0.0,0.7,0.2}
\usepackage[pdftex,colorlinks=true,citecolor=myBlue,linkcolor=myBlue]{hyperref}

\usepackage{comment,enumerate,multicol,xspace}

  \newcounter{mnote}
  \let\oldmarginpar\marginpar
    \renewcommand\marginpar[1]{\-\oldmarginpar[\raggedleft\footnotesize #1]%
    {\raggedright\footnotesize #1}}

\newtheorem{theorem}{Theorem}[section]
\newtheorem{lemma}[theorem]{Lemma}

\newtheorem{remark}[theorem]{Remark}

\newcommand{\dt}{\,\Delta t}

\newcommand{\dd}{\,{\rm d}}

\newcommand{\bs}{\boldsymbol}

\newcommand{\la}{\langle}
\newcommand{\ra}{\rangle}

\newcommand{\vertiii}[1]{{\left\vert\kern-0.25ex\left\vert\kern-0.25ex\left\vert #1 
    \right\vert\kern-0.25ex\right\vert\kern-0.25ex\right\vert}}

\newcommand{\curl}{{\rm curl\,}}
\renewcommand{\div}{\operatorname{div}}
\newcommand{\grad}{{\rm grad\,}}
\DeclareMathOperator*{\tr}{tr}
\DeclareMathOperator*{\rot}{rot}

\usepackage[margins]{trackchanges}
\renewcommand{\initialsOne}{chen}

\begin{document}
\title{Energy-preserving mixed finite element methods for the Hodge wave equation}
\author{Yongke Wu and Yanhong Bai}
\date{\today}
\thanks{Y. Wu was supported by the National Natural Science Foundation of China (11971094 and 11501088). Y. Bai was supported by the National Natural Science Foundation of China (11701481).}

\address[Y.~Wu]{School of Mathematical Sciences, University of Electronic Science and Technology of China, Chengdu 611731, China.}
\email{wuyongke1982@sina.com}

\address[Y.~Bai]{School of Science, Xihua University, Chengdu 610039, China.}
\email{baiyanhong1982@126.com}


\begin{abstract}
Energy-preserving numerical methods for solving the Hodge wave equation is developed in this paper. Based on the de Rham complex, the Hodge wave equation can be formulated as a first-order system and mixed finite element methods using finite element exterior calculus is used to discretize the space. A continuous time Galerkin method, which can be viewed as a modification of the Crank-Nicolson method, is used to discretize the time which results in a full discrete method preserving the energy exactly when the source term is vanished. A projection based operator is used to establish the optimal order convergence of the proposed methods. Numerical experiments are present to support the theoretical results. 
\end{abstract}


\maketitle

\section{Introduction}

We consider energy-preserving numerical methods for solving the Hodge wave equation, the hyperbolic equation in $\mathbb R^n$ associated to the Hodge Laplacian of differential $k$-forms for $0\leq k \leq n$. The initial-boundary value problem we study is: Find $u:\ (0,T]\mapsto H_{0}\Lambda^k(\Omega)$ satisfying 
\begin{align}\label{eq:hodgewave-1}
u_{tt} + (\dd\delta +\delta \dd) u  & = f \qquad \text{in }\Omega \times (0,T], 
\end{align}
with homogeneous boundary conditions
\begin{align} 
\label{eq:hodgewave-2}
\tr(u) = 0,\quad \tr(\star \dd u)& = 0\qquad \text{on } \partial\Omega \times (0,T], 
\end{align}
and initial conditions
\begin{align}
\label{eq:hodgewave-3}
u(\cdot,0) = u_0(\cdot),\quad u_t(\cdot,0) & = u_1(\cdot)\qquad  \text{in } \Omega.
\end{align}
Here $\Omega \subset \mathbb R^n$ is a domain homomorphism to a ball with piecewise smooth and Lipschitz boundary. The unknown $u$ is a time dependent differential $k$-form on $\Omega$, $u_t$ and $u_{tt}$ denote its partial derivatives with respect to time variable, and $\dd$, $\delta$, $\star$, and $\tr$ denote exterior derivative, co-derivative, Hodge star, and the trace operator, respectively; see Section \ref{sec:pr} for precise definitions. We assume that $T$ is a finite positive real number denoting the ending time. 
%

Many physical problems can be described by \eqref{eq:hodgewave-1}, such as the mathematical models of sound waves ($n = 3$ and $k = 0$), electromagnetic waves ($n = 3$ and $k = 1$), structural vibration ($n = 3$ and $k = 2$) and so on. There are many theoretical analyses of finite element method for \eqref{eq:hodgewave-1} in the special case $n = 2$ or $3$ and $k = 0$ or $k = n-1$; see \cite{Dupont1973L2,Baker1976Error,geveci1988application,Cowsat1990A,French1996A,Rivi2003A,Glowinski2006A,Jenkins2007Numerical,Chung2009OPTIMAL,Kirby_2014,Griesmaier2014} and the references therein. The pioneer work on mixed finite element methods \cite{Brezzi;Fortin1991} for the general form of the Hodge wave equation \eqref{eq:hodgewave-1} can be found in Quenneville-B\'elair's Ph.~D.~thesis \cite{Quenneville-Belair2015}; see also \cite{Arnold2018}. In this work, he has presented 
(1) the abstract Hodge wave equation in the mixed form,
(2) the semi-discretization in space for solving the Hodge wave equation
(3) the existence and uniqueness of the solution for the semi-discretization in space, 
(4) the error estimates in the $\|\cdot\|_{L^\infty(L^2)}$ norm for the semi-discretization in space based on the elliptic projection operator. 

In the present work, we shall give more thorough analysis of the mixed finite element method developed in \cite{Quenneville-Belair2015,Arnold2018}. Introduce a $(k - 1)$-form $\sigma = \delta u$ and a $(k+1)$-form $\omega = \dd u$ with standard modification for $k = 0$ or $k = n$, and a $k$-form $\mu = u_{t}$.
The first order formulation of \eqref{eq:hodgewave-1} reads as: find $\sigma \in H_{0}\Lambda^{-}$, $\mu \in H_{0}\Lambda$, and $\omega \in H_{0}\Lambda^{+}$ 
such that
\begin{align}
\label{eq:1}
\langle \sigma_{t},\tau\rangle - \langle\dd^{-}\tau,\mu\rangle  & = 0\qquad\qquad\forall~~\tau \in H_{0}\Lambda^{-}, \\
\label{eq:2}
\langle \mu_{t},v\rangle + \langle \dd^{-}\sigma,v\rangle + \langle \omega,\dd v\rangle & = \langle f,v\rangle\,\qquad \forall~~v\in H_{0}\Lambda, \\
\label{eq:3}
\langle \omega_{t},\phi\rangle - \langle \dd \mu,\phi \rangle & = 0\qquad\qquad \forall~~\phi \in H_{0}\Lambda^{+},
\end{align}
with initial conditions 
$$
\sigma_{0} = \delta u_{0},\quad \mu_{0} = u_{1},\quad \omega_{0} = \dd u_{0}.
$$
Comparing with \cite{Quenneville-Belair2015}, the main contributions of this paper are as follows. Firstly, we use the skew-symmetric property of the formulation \eqref{eq:1}-\eqref{eq:3} to get  the following energy estimates
\begin{align}
\label{Establity} \sup\limits_{0 \leq t \leq T}E(t) & \leq E(0) + 2\int_0^T\|f(\cdot,s)\|\dd s, \\
\label{Hstablity} \sup\limits_{0 \leq t \leq T} H(t) & \leq H(0) + 4\|f\|_{L^\infty(L^2)} + 2\int_0^T\|f_t(\cdot,s)\|\dd s,
\end{align}
with
\begin{align*}
E(t) &= (\|\sigma(\cdot,t)\|^2 + \|\mu(\cdot,t)\|^2 + \|\omega(\cdot,t)\|^2)^{1/2},\\
H(t) & = (\|\dd^-\sigma(\cdot,t)\|^2 + \|\dd\mu(\cdot,t)\|^2 + \|\delta \mu(\cdot,t)\|^2 + \|\delta^+\omega(\cdot,t)\|^2)^{1/2}.
\end{align*}
These energy estimates imply the existence and uniqueness of solution for \eqref{eq:1}-\eqref{eq:3}; see Remark \ref{re:uni-exi-solu}. When \eqref{eq:1}-\eqref{eq:3} is self-conserve, i.e., $f = 0$, the inequality \eqref{Establity}-\eqref{Hstablity} become equalities which implies the energies $E$ and $H$ are preserved exactly; see Remark \ref{re:en-pre-exact}. 
Due to the structure preserving properties of the finite element exterior calculus (FEEC)~\cite{Arnold;Falk;Winther2006,Arnold;Falk;Winther2010,Arnold2018}, the semi-discretization in space also inherit the skew-symmetric property of the spatial differential terms, and thus the energy conservation is preserved naturally. We then use the continuous time Galerkin method~\cite{French1996A} to give unconditioned energy conservation schemes. Here 
we follow the approach in  \cite{geveci1988application,French1996A,Kirby_2014}, where the energy estimates has been derived for scalar wave equations but not for Hodge wave equations. As we know, energy conservation numerical schemes can have a crucial influence on the quality of the numerical simulations. Especially, in long-time simulations, energy-preserving can have a dramatic effect on stability and global error growth. 

Secondly, we obtain the optimal convergence order of the error estimates in both $L^2$-norm and $\|\mathcal A(\cdot)\|_{L^\infty(L^2)}$-norm for the semi- and full-discrete mixed finite element methods, where $\mathcal A$ is a skew-symmetric operator defined in Section \ref{sec:pr}. Such result has been derived for scalar wave equation \cite{geveci1988application,French1996A,Kirby_2014} but generalization to general Hodge wave equation is non-trivial. Technically, the canonical interpolation operators $\pi_{h}$ used in \cite{geveci1988application,French1996A,Kirby_2014} cannot be commutated with the discrete co-derivative operator $\delta_{h}$, and the $L^{2}$ projection operator $Q_{h}$ cannot be commutated with the exterior derivative operators $\dd$. Using these standard operators in the convergence analysis will lead to the lost of the convergence order. To overcome this difficulty, we choose a projection based interpolation operator $I_{h}$ briefly mentioned in \cite[Proposition 5.44]{Christiansen2011} and redefine it based on the Hodge decomposition. Such projection based operators has been introduced for $H^1, H(\curl)$ and $H(\div)$ spaces in \cite{Demkowiz2000,Demkowiz2003,Demkowiz2005,Oden1989}, where the authors have proved that these projection based operators made the de Rahm diagram commute and had the quasi-optimal interpolation error bound for $hp$ finite element spaces. Although this projection-based quasi-interpolation operator is not new, the properties we are going to prove are not fully explored in the literature. Specifically, we shall prove that (1) $I_h$ is commuted with $\delta_{h}$, (2) $I_{h}$ is stable in both $\|\dd(\cdot)\|$ and $\|\delta_{h}(\cdot)\|$ norms, (3) $I_{h}$ is the $L^{2}$ orthogonal projection to the space $\mathfrak Z_{0,h}$, (4) $I_{h}$ is an orthogonal projection operator with respect to the inner-product $\la\dd(\cdot),\dd(\cdot)\ra$, (5) $I_{h}$ has the same approximation properties as the classical interpolation operators; see Lemma \ref{eq:I-h-orth} - \ref{lem:app-I-h}. 
By using these properties of the projection-based operator $I_h$, we get the optimal error estimates for both the semi- and full- discretization with respect to both $\|\cdot\|_{L^\infty(L^2)}$ and $\|\mathcal A(\cdot)\|_{L^\infty(L^2)}$ norms (the detail definition of these norms can be found in Section \ref{sec:semi}), while recall that \cite{Quenneville-Belair2015} only give the error estimates of $\|\cdot\|_{L^\infty(L^2)}$ for the semi-discretization in space and as the line of Quenneville's proof, it seems difficulty to get the error estimate of the energy norm $\|\mathcal A(\cdot)\|_{L^\infty(L^2)}$. But the control of the energy norm is very important, since the $L^2$-norm is possible small but the energy norm is larger due to the small oscillation in the error. Furthermore our error estimate, comparing with \cite{Quenneville-Belair2015}  is robust to $T$ in the sense that the factor $T$ is absent on the error estimates; see Theorem \ref{the:err-semi-L2}, \ref{the:err-semi-A}, \ref{the:err-full-L2} and \ref{the:err-full-A}. Such error estimates imply that our algorithms are robust for long time problems and the numerical experiment supports this result; see Table \ref{EX1-3}.

What remains of this paper is organized as follow. In Section 2 we introduce the required background on finite element exterior calculus (FEEC) and the Hodge wave equation. We obtain the mixed formulation of the Hodge wave equation and get the energy conservation estimates. Section 3, we briefly introduce the finite element spaces on $k$-forms, give the semi-discrete form of the Hodge wave equation, introduce a projection-based quasi-interpolation operator and explore properties of this operator, obtain the energy estimates of the semi-discrete form, and get the optimal error estimates of the semi-discrete form. In section 4, the full-discrete form of the Hodge wave equation is obtained, the energy estimates and the optimal error estimates are obtained. Section 5 give some numerical experiments to confirm our theoretical results.  

Throughout this paper, $i$, $h$ and $\dt$ denote the time level, the mesh size and the time step size, respectively. The capital $C$ may be different in different places, denotes a positive constant which is independent on $i$, $h$ and $\dt$. We denote by $\|\cdot\|_{m,p}$ the norm of the classical Sobolev spaces $W^{m,p}\Lambda^{k}(\Omega)$, $1\leq p \leq \infty$ and $0\leq k \leq n$. If $ p = 2$, we write $\|\cdot\|_{m,p}$ simply as $\|\cdot\|_{m}$ and denote by $|\cdot|_{m}$ the semi-norm in $W^{m,2}\Lambda^{k}(\Omega)$. In addition, for any Sobolev space $Y$, we define the space $L^{p}([a,b],Y)$ with norm $\|f\|_{L^{p}(Y)} = \left(\int_{a}^{b}\|f(\cdot,t)\|_{Y}^{p} \dd t \right)^{1/p} $, and if $p = \infty$, the integral is replaced by the essential supremum.

\section{Preliminaries}\label{sec:pr}
In this section, we follow the convention of \cite{Arnold;Falk;Winther2006,Arnold;Falk;Winther2010,Arnold2018} to introduce necessary background of finite element exterior calculus. Then, we introduce the Hodge wave equation and its mixed formulation. Finally, we get the energy conservation estimates for this mixed form.

\subsection{de Rham complex}
Let $\Omega \subset \mathbb R^n$ ($n \geq 2$) be a bounded Lipschitz domain. For a given integer $0\leq k \leq n$, $\Lambda^k(\Omega)$ represents the linear space of all smooth $k$-forms on $\Omega$. For any $\omega \in \Lambda^k(\Omega)$, $\omega$ can be written as
$$
\omega = \sum\limits_{1\leq \sigma_1 < \cdots<\sigma_k \leq n} a_\sigma \dd x^{\sigma_1} \wedge \cdots \wedge\dd x^{\sigma_k},
$$
with $a_\sigma \in C^\infty(\Omega)$ and $\wedge$ the wedge product. As $\Omega$ is a flat domain in $\mathbb R^n$, we can identify each tangent space of $\Omega$ with $\mathbb R^n$. Given an $\omega \in \Lambda^k(\Omega)$ and vectors $v_1,~v_2,\cdots,~v_k \in \mathbb R^{n}$, we have that the map $\bs x \in \Omega \mapsto \omega_{\bs x}(v_{1},v_{2},\cdots,v_{k}) \in \mathbb R$ is a smooth map (infinitely differentiable).


We define the exterior derivative  $\dd^{k}:~\Lambda^{k}(\Omega) \rightarrow \Lambda^{k+1}(\Omega)$ as
$$
\dd^{k}\omega_{x}(v_{1},v_{2},\cdots,v_{k+1}) = \sum\limits_{j = 1}^{k+1}(-1)^{j+1}\partial_{v_{j}}\omega_{x}(v_{1},\cdots,\hat{v}_{j},\cdots,v_{k+1}),
$$
where the hat is used to indicate a suppressed argument. By the definition of $\dd^{k}$, it is easy to see that $\dd^{k}$ is a sequence of differential operators satisfying that the range of $\dd^{k}$ lies in the domain of $\dd^{k+1}$, i.e., $\dd^{k+1} \circ \dd^{k} = 0$ for $k = 0,1,\cdots,n-1$. For convenience of notation, we shall skip the superscript $k$ if there is no confusion. 

Let $\text{vol}$ be the unique volume form in $\Lambda^{k}(\Omega)$, define the $L^{2}$-inner product of any two differential $k$-forms on $\Omega$ as the integral of their pointwise inner product:
$$
\la\omega,\mu\ra = \int_{\Omega}\la\omega_{x},\mu_{x}\ra\text{vol}.
$$
The completion of $\Lambda^{k}(\Omega)$ under the corresponding norm defines the Hilbert space $L^{2}\Lambda^{k}(\Omega)$. The domain of the exterior derivative $\dd^{k}$ can be enlarged to
$$
H\Lambda^{k}(\Omega) = \{\omega \in L^{2}\Lambda^{k}(\Omega):~\dd\omega \in L^{2}\Lambda^{k+1}(\Omega) \}.
$$
$H\Lambda^{k}(\Omega)$ is a Hilbert space with inner product $\langle\omega,\mu\rangle + \langle \dd\omega,\dd\mu\rangle$ and associated graph norm $\|\cdot\|_{H\Lambda}$. The de Rham complex
\begin{equation}\label{eq:deRham-1}
\begin{CD}
H\Lambda^{0}(\Omega)@>{\dd}>>  H\Lambda^{1}(\Omega)  @>{\dd}>> \cdots @>{\dd}>> H \Lambda^{n-1}(\Omega)@>{\dd}>> H\Lambda^{n}(\Omega) 
\end{CD}
\end{equation}
is then bounded in the sense that $\dd:~H\Lambda^{k}(\Omega) \rightarrow H\Lambda^{k+1}(\Omega)$ is a bounded operator.

For any smooth manifold $M$ and any $x \in M$, we use $T_{x}M$ to denote the tangential space of $M$ at $x$. For any smooth $k$-form $\omega \in \Lambda^{k}(\Omega)$, we define $\tr\omega \in \Lambda^{k}(\partial M)$ as
$$
\tr\omega(v_{1},v_{2},\cdots,v_{k}) = \omega(v_{1},v_{2},\cdots,v_{k})
$$
for tangential vectors $v_{i} \in T_{x}\partial M \subset T_{x}M$ ($i = 1,2,\cdots,k$). This operator can be extended continuous to Lipschitz domain $\Omega$, also denote by $\tr:~H^{1}\Lambda^{k}(\Omega) \rightarrow H^{1/2}\Lambda^{k}(\partial\Omega)$ and $\tr:~H\Lambda^{k}(\Omega) \rightarrow H^{-1/2}\Lambda^{k}(\partial \Omega)$. Define
\begin{align*}
H_{0}\Lambda^{k}(\Omega) & = \{ \omega \in H\Lambda^{k}(\Omega):~\tr\omega = 0\text{ on }\partial\Omega \}, \\
H_{0}^{1}\Lambda^{k}(\Omega) & = \{ \omega \in H^{1}\Lambda^{k}(\Omega):~\tr\omega = 0\text{ on }\partial\Omega \}.
\end{align*}
In the following sections, we will focus on the de Rham complex with homogeneous trace
\begin{equation}\label{eq:deRham-prime}
\begin{CD}
H_{0}\Lambda^{0}(\Omega)@>{\dd}>> H_{0}\Lambda^{1} (\Omega) @>{\dd}>>\cdots @>{\dd}>> H_0\Lambda^{n-1} (\Omega)@>{\dd}>> H_0\Lambda^n(\Omega) 
\end{CD}
\end{equation}

In order to define the dual complex, we start with the Hodge star operator $\star:~\Lambda^{k}(\Omega) \rightarrow \Lambda^{n-k}(\Omega)$,
$$
\int_{\Omega}\omega \wedge \mu = \langle \star\omega,\mu\rangle,\qquad \forall~~\omega \in \Lambda^{k}(\Omega),~~\mu\in \Lambda^{n-k}(\Omega).
$$
The coderivative operator $\delta^{k}:~\Lambda^{k}(\Omega) \rightarrow \Lambda^{k-1}(\Omega)$ is defined as
$$
\delta^{k} \omega = (-1)^{k(n-k+1)}\star \dd^{n-k} \star \omega.
$$
$\dd^{k-1}$ and $\delta^{k}$ are related by the Stokes theorem
$$
\langle \dd\omega,\mu\rangle = \langle \omega,\delta\mu\rangle + \int_{\partial\Omega} \tr\omega \wedge \tr(\star\mu), \qquad \omega \in \Lambda^{k-1}(\Omega),~~\mu \in \Lambda^{k}(\Omega).
$$
%
%
We define the spaces
\begin{align*}
H^{*}\Lambda^{k}(\Omega) & = \{ \omega \in L^{2}\Lambda^{k}(\Omega):~\delta \omega \in L^{2}\Lambda^{k-1} (\Omega)\},\\
H_{0}^{*}\Lambda^{k}(\Omega) & = \{ \omega \in H^{*}\Lambda^{k}(\Omega):~\tr \star\omega = 0\text{ on }\partial\Omega\}.
\end{align*}
Treat $\dd:~H_{0}\Lambda^{k}(\Omega) \subset L^{2}\Lambda^{k}(\Omega) \rightarrow L^{2}\Lambda^{k+1}(\Omega)$ as an unbounded and densely defined operator. Then Stokes theorem implies that $\delta:~H^{*}\Lambda^{k+1}(\Omega)\subset L^{2}\Lambda^{k+1}(\Omega) \rightarrow L^{2}\Lambda^{k}(\Omega)$ is the adjoint of $\dd$ as
\begin{equation}
\label{eq:delta}
\langle \dd \omega,\mu\rangle = \langle\omega,\delta\mu\rangle,\qquad\forall~~\omega \in H_{0}\Lambda^{k}(\Omega),~~\mu\in H^{*}\Lambda^{k+1}(\Omega).
\end{equation}
We have a dual sequence of \eqref{eq:deRham-prime}
\begin{equation}\label{eq:deRham}
\begin{CD}
 H^{*}\Lambda^0(\Omega) @<{\delta}<<   H^{*}\Lambda^1(\Omega)    @<{\delta}<<\cdots @< {\delta}<<   H^{*}\Lambda^{n-1}(\Omega) @<{\delta}<< H^{*}\Lambda^n(\Omega).
\end{CD}
\end{equation}

Let $\mathfrak Z_{0}^{k}$ be the kernel of $\dd$ in the space $H_{0}\Lambda^{k}(\Omega)$, then $\mathfrak Z_{0}^{k}$ can be decomposed as $\mathfrak Z_{0}^{k} = \mathfrak B_{0}^{k} \oplus^{\bot_{L^{2}}} \mathfrak H_{0}^{k}$, where $\mathfrak B_{0}^{k}$ is the range of $\dd^{k-1}$, i.e., $\mathfrak B_{0}^{k} = \dd (H_{0}\Lambda^{k-1}(\Omega))$ and $\mathfrak H_{0}^{k}$ is the space of harmonic forms, i.e., $\mathfrak H_{0}^{k} = \{\omega\in H_{0}\Lambda^{k}(\Omega) \cap H^{*}\Lambda^{k}(\Omega):~~\dd\omega = 0 \text{ and } \delta\omega = 0\}$, $\oplus^{\bot_{L^{2}}}$ means that the decomposition is orthogonal in the sense of the $L^{2}$-inner product. The following Hodge decomposition has been established in \cite[page 22]{Arnold;Falk;Winther2006}:
$$
L^{2}\Lambda^{k}(\Omega) = \mathfrak B_{0}^{k}\oplus^{\bot_{L^{2}}} \mathfrak H_{0}^{k} \oplus^{\bot_{L^{2}}} \delta H^{*}\Lambda^{k+1}(\Omega).
$$
Denote $\mathfrak K^{k}$ as the $L^{2}$ orthogonal complement of $\mathfrak Z_{0}^{k}$ in $H_{0}\Lambda^{k}(\Omega)$, i.e., $\mathfrak K^{k} = H_{0}\Lambda^{k}(\Omega) \cap \delta H^{*}\Lambda^{k+1}(\Omega)$. Then we have the Hodge decomposition of $H_{0}\Lambda^{k}(\Omega)$:
\begin{equation}\label{eq:hodge}
H_{0}\Lambda^{k}(\Omega) = \mathfrak Z_{0}^{k} \oplus^{\bot_{L^{2}}}\mathfrak K^{k}= \mathfrak B_{0}^{k} \oplus^{\bot_{L^{2}}} \mathfrak H_{0}^{k} \oplus^{\bot_{L^{2}}} \mathfrak K^{k}.
\end{equation}
It should be point out that when $k =0$, we have $\mathfrak Z_{0}^{0} = \{0\}$ and $\mathfrak K^{-1} = \{0\}$. When $k = n$, we have $\mathfrak K^{n} = \{0\}$.

In the following sections, when spaces of the consecutive differential forms are involved, we use the short sequences
\begin{equation}
\label{eq:shortsqe}
\begin{CD}
H_{0}\Lambda^{-}(\Omega) @>{\dd^{-}} >> H_{0}\Lambda(\Omega) @>{\dd}>> H_{0}\Lambda^{+}(\Omega)
\end{CD}
\end{equation}
or the one with the Hodge decomposition
\begin{equation}
\label{eq:shortsqe_hodge}
\begin{CD}
\mathfrak B_{0}^{-} \oplus^{\bot_{L^{2}}} \mathfrak H_{0}^{-} \oplus^{\bot_{L^{2}}} \mathfrak K^{-} @>{\dd^{-}} >> \mathfrak B_{0} \oplus^{\bot_{L^{2}}} \mathfrak H_{0} \oplus^{\bot_{L^{2}}} \mathfrak K @>{\dd} >>\mathfrak B_{0}^{+} \oplus^{\bot_{L^{2}}} \mathfrak H_{0}^{+} \oplus^{\bot_{L^{2}}} \mathfrak K^{+}.
\end{CD}
\end{equation}

In this paper, we consider the domain $\Omega$ with zero Betti numbers, namely, we impose the following assumption on the domain $\Omega$:
\begin{description}
  \item[(A)] We assume that $\Omega$ is simple in the sense that $\dim\mathfrak H_{0}^{k} = 0$ for all $1\leq k \leq n-1$.
\end{description} 

\subsection{The Hodge wave equation}
The Hodge wave equation reads as given $f:~(0,T)\mapsto L^{2}\Lambda^{k}$, find $u \in H^{2}((0,T),H_{0}\Lambda^{k} \cap H^{*}\Lambda^{k})$ such that
\begin{equation}\label{eq:HodgeWave}
u_{tt} + \mathcal Lu = f\qquad \text{in }\Omega,
\end{equation} 
where $\mathcal L = \dd^{-}\delta + \delta^{+}\dd$ is called the Hodge Laplacian operator \cite{Arnold;Falk;Winther2010}, with the initial conditions
\begin{equation}\label{eq:initialHodgeWave}
u(\cdot,0) = u_{0}(\cdot),\qquad u_{t}(\cdot,0) = u_{1}(\cdot).
\end{equation}


For easy to preserve the energy exactly, we will use mixed method to discrete \eqref{eq:HodgeWave}. Introduce a $(k - 1)$-form $\sigma = \delta u$ and a $(k+1)$-form $\omega = \dd u$ with standard modification for $k = 0$ or $k = n$, and a $k$-form $\mu = u_{t}$.
The mixed formulation \cite{Quenneville-Belair2015} of the Hodge wave equation \eqref{eq:HodgeWave} is: given $f \in L^{2}((0,T),L^{2}\Lambda)$, find $(\sigma,\mu,\omega):~(0,T] \mapsto H_{0}\Lambda^{-}\times H_{0}\Lambda \times H_{0}\Lambda^{+} := \bs W$ such that
\begin{align}
\label{eq:mix1}
\langle \sigma_{t},\tau\rangle - \langle\dd^{-}\tau,\mu\rangle  & = 0\qquad\qquad\forall~~\tau \in H_{0}\Lambda^{-}, \\
\label{eq:mix2}
\langle \mu_{t},v\rangle + \langle \dd^{-}\sigma,v\rangle + \langle \omega,\dd v\rangle & = \langle f,v\rangle\,\qquad \forall~~v\in H_{0}\Lambda^{k}, \\
\label{eq:mix3}
\langle \omega_{t},\phi\rangle - \langle \dd \mu,\phi \rangle & = 0\qquad\qquad \forall~~\phi \in H_{0}\Lambda^{+},
\end{align}
with initial conditions
$$
\sigma(\cdot,0) = \delta u_{0},\quad \mu(\cdot,0) = u_{1}(\cdot),\quad \omega(\cdot,0) = \dd u_{0}(\cdot).
$$
Denoted by 
$$
\mathcal A = \begin{pmatrix}
0 & \delta & 0 \\ -\dd^{-} & 0 & - \delta^{+} \\ 0 & \dd & 0
\end{pmatrix}.
$$
 
The existence of solutions for the mixed formulation \eqref{eq:mix1}-\eqref{eq:mix3} can be found in \cite{Quenneville-Belair2015} and it can also be obtained by Picard Theorem since the operator 
$$
\mathcal A:~H_{0}\Lambda^{-} \times (H_{0}\Lambda \cap H^{*}\Lambda) \times (H_{0}\Lambda^{+} \cap H^{*}\Lambda^{+}) \rightarrow L^{2}\Lambda^{-} \times L^{2} \Lambda \times L^{2}\Lambda^{+}
$$
is bounded. To prove the uniqueness of the solution, we need the energy estimates. We introduce a basic inequality.
\begin{lemma}{{\rm (}\cite[Lemma 1]{Kirby_2014}{\rm )}}\label{lem:greatwall}
Suppose that a real number $x$ satisfies the quadratic inequality 
$$
x^{2} \leq \gamma^{2} + \beta x
$$
for $\beta,~\gamma \geq 0$ and $\beta^{2} + \gamma^{2} >0$. Then
$$
x \leq \beta + \gamma.
$$
\end{lemma}

We define two energies of the mixed formulation \eqref{eq:mix1}-\eqref{eq:mix3} as
$$
E(t) = \left(\|\sigma(\cdot,t)\|^{2} + \|\mu(\cdot,t)\|^{2} + \|\omega(\cdot,t)\|^{2} \right)^{1/2}
$$
and 
$$
H(t) =\left( \|\dd^{-}\sigma(\cdot,t)\|^{2} + \|\dd \mu(\cdot,t)\|^{2}  + \|\delta\mu(\cdot,t)\|^{2}+ \|\delta^{+}\omega(\cdot,t)\|^{2} \right)^{1/2}.
$$
We have the following energy estimates.
\begin{theorem}\label{the:energy_continuous}
Let $\bs u = (\sigma,\mu,\omega)^{\intercal}\in \bs W$ be the solution of the mixed formulation \eqref{eq:mix1}-\eqref{eq:mix3}. Provided $f \in L^{1}((0,T),L^{2}\Lambda)$, we have the energy bound
\begin{equation}\label{eq:energy1}
\sup\limits_{0\leq s\leq T} E(t) \leq E(0) + 2\int_{0}^{T}\|f(\cdot,s)\| \dd s.
\end{equation}
Furthermore, if $f\in W^{1,1}((0,T),L^{2}\Lambda)$, we have the bound
\begin{equation}\label{eq:energy2}
\sup\limits_{0\leq t \leq T}H(t) \leq H(0) + 4\|f\|_{L^\infty(L^2)} +2\int_{0}^{T} \|f_{t}(\cdot,s)\|\dd s.
\end{equation}
When $f = 0$, the inequalities become equalities and thus we have the energy conservation 
$$
E(t) = E(0),  H(t) = H(0), \quad \forall t>0.
$$

\end{theorem}
\begin{proof}
Taking $\tau = \sigma$, $v = \mu$ and $\phi = \omega$ in \eqref{eq:mix1} - \eqref{eq:mix3}
and adding them together, we obtain
$$
\frac{1}{2}\frac{\dd }{\dd t} E^{2}(t) = \langle f,\mu\rangle.
$$
Integrate the above equation on the interval $(0,s)$, for any $s \in (0,T]$, we have
\begin{align*}
E^{2}(s) & = E^{2}(0) + 2 \int_{0}^{s}\langle f,\mu\rangle \dd t \\
& \leq E^{2}(0) + 2 \sup\limits_{0 \leq t \leq T}E(t) \int_{0}^{T}\|f\| \dd t.
\end{align*}
Then \eqref{eq:energy1} follows by Lemma \ref{lem:greatwall}.

Taking $\tau = \delta \mu_{t}$, $v = -\dd^{-}\sigma_{t} -\delta^{+}\omega_{t}$ and $\phi = \dd\mu_{t}$ in \eqref{eq:mix1}-\eqref{eq:mix3}, we have
\begin{align*}
\langle\dd^{-}\sigma_{t},\mu_{t}\rangle - \frac{1}{2}\frac{\dd}{\dd t}\|\delta\mu\|^{2} & = 0, \\
-\langle \mu_{t},\dd^{-}\sigma_{t}\rangle - \langle \dd\mu_{t},\omega_{t}\rangle -\frac{1}{2}\frac{\dd }{\dd t}\|\dd^{-}\sigma\|^{2} - \frac{1}{2}\frac{\dd }{\dd t}\|\delta^{+}\omega\|^{2} & = -\langle f ,\dd^{-}\sigma_{t} + \delta^{+}\omega_{t} \rangle,\\
\langle \omega_{t},\dd\mu_{t}\rangle - \frac{1}{2}\frac{\dd}{\dd t}\|\dd\mu\|^{2} & = 0.
\end{align*}
Add the above equations together, we obtain
$$
\frac{1}{2}\frac{\dd}{\dd t} H^{2}(t) = \langle f,\dd^{-}\sigma_{t} + \delta^{+}\omega_{t} \rangle.
$$
Pick any $0 \leq s \leq T$ and integrate from $0$ to $s$ to obtain
\begin{align*}
H^{2}(s) & = H^{2}(0) + 2 \int_{0}^{s}\langle f,\dd^{-}\sigma_{t} + \delta^{+}\omega_{t}\rangle \dd t\\
& = H^{2}(0) + 2\langle f(\cdot,s),\dd^{-}\sigma(\cdot,s) + \delta^{+}\omega(\cdot,s)\rangle - 2\langle f(\cdot,0),\dd^{-}\sigma(\cdot,0) + \delta^{+}\omega(\cdot,0)\rangle\\
&\quad - 2\int_{0}^{s}\langle f_{t},\dd^{-}\sigma + \delta^{+}\omega\rangle \dd t \\
& \leq H^{2}(0) + \sup\limits_{0\leq t\leq T} H(t) \left(4\|f\|_{L^\infty(L^2)} + 2\int_{0}^{T} \|f_{t}\|\dd t \right).
\end{align*}
Then the desired inequality \eqref{eq:energy2} follows from Lemma \ref{lem:greatwall}. 
\end{proof}

\begin{remark}\label{re:en-pre-exact}
When the source term $f$ of \eqref{eq:HodgeWave} equal $0$, i.e., \eqref{eq:HodgeWave} is a self-conserve system, Theorem \ref{the:energy_continuous} implies that the mixed form \eqref{eq:mix1}-\eqref{eq:mix3} preserves the energies $E$ and $H$ exactly.
\end{remark}
\begin{remark}\label{re:uni-exi-solu}
Theorem \ref{the:energy_continuous} implies the uniqueness of solution of \eqref{eq:mix1}-\eqref{eq:mix3} in the space $\bs W$. Together with the existence of solutions in the space $H_{0}\Lambda^{-} \times (H_{0}\Lambda \cap H^{*}\Lambda) \times (H_{0}\Lambda^{+}\cap H^{*}\Lambda^{+})$, we obtain that \eqref{eq:mix1}-\eqref{eq:mix3} have a unique solution $\bs u = (\sigma,\mu,\omega)^{\intercal}$ in the space $H_{0}\Lambda^{-} \times (H_{0}\Lambda \cap H^{*}\Lambda) \times (H_{0}\Lambda^{+}\cap H^{*}\Lambda^{+})$ and for any $t \in (0,T]$ satisfying 
\begin{equation}\label{eq:hodge-operator-form}
\la\bs u_{t},\bs v\ra + \la\mathcal A\bs u,\bs v\ra = \la\bs F,\bs v\ra\qquad\forall~~\bs v \in \bs W,
\end{equation}
with $\bs F = (0,~f,~0)^{\intercal}$.
\end{remark}

\section{Semi-discretization of the Hodge wave equation}
\label{sec:semi}

In this section, we will introduce mixed finite element methods developed in \cite{Quenneville-Belair2015,Arnold2018} for the spatial discretization of the Hodge wave equation \eqref{eq:HodgeWave}, and give the energy estimates and optimal error estimates.

\subsection{Finite element spaces}
Let $\mathcal T_{h}$ be a shape regular triangulation of $\Omega$. For each $n$-simplex $K \in \mathcal T_{h}$, we define $h_{K} = |K|^{1/n}$ and $h = \max\limits_{K \in \mathcal T_{h}} h_{K}$. For completeness, we briefly introduce the construction of finite element spaces following \cite{Arnold;Falk;Winther2006,Arnold2018}.

Denote $\mathcal P_{r}(\mathbb R^{n})$ as the space of polynomials in $n$ variables of degree at most $r$ and $\mathcal H_{r}(\mathbb R^{n})$ as the space of homogeneous polynomial functions of degree $r$. Spaces of polynomial differential forms $\mathcal P_{r}\Lambda^{k}(\mathbb R^{n})$ and $\mathcal H_{r}(\mathbb R^{n})$ can be defined by using the corresponding polynomial as the coefficients. We will suppress $\mathbb R^{n}$ from the notation for simplicity. For each integer $r \geq n$, we have the polynomial subcomplex of the de Rham complex
$$
\begin{CD}
0 @>{}>>  \mathcal P_{r}\Lambda^0 @>{\dd}>> \mathcal P_{r-1}\Lambda^1    @>{\dd}>> \cdots @>{\dd}>> \mathcal P_{r-n}\Lambda^{n} @>{}>> 0. 
\end{CD}
$$
Given a point $x \in \mathbb R^{n}$, treat $x$ as a vector in the tangential space $T_{x}\mathbb R^{n}$ and define the Koszul operator $\kappa:~\Lambda^{k}(\mathbb R^{n}) \rightarrow \Lambda^{k-1}(\mathbb R^{n})$ as 
$$
(\kappa\omega)_{x}(v_{1},v_{2},\cdots,v_{k-1}) = \omega_{x}(x,v_{1},v_{2},\cdots,v_{k-1}).
$$
This $\kappa$ satisfying the identity $\kappa\dd + \dd\kappa = (k+r){\rm id}$ \cite[Theorem 3.1]{Arnold;Falk;Winther2006} on the space $\mathcal H_{r}\Lambda^{k}$ and there is a direct sum
$$
\mathcal H_{r}\Lambda^{k} = \kappa \mathcal H_{r-1}\Lambda^{k+1} \oplus \dd\mathcal H_{r+1}\Lambda^{k-1}.
$$
Based on the decomposition, the incomplete polynomial differential form can be introduced as 
$$
\mathcal P_{r}^{-}\Lambda^{k} = \mathcal P_{r-1}\Lambda^{k} + \kappa \mathcal H_{r-1}\Lambda^{k+1}
$$
and, for $r\geq 1$, have the following subcomplex of the de Rham complex
$$
\begin{CD}
0 @>{}>>  \mathcal P_{r}^{-}\Lambda^0 @>{\dd}>> \mathcal P_{r}^{-}\Lambda^1    @>{\dd}>> \cdots @>{\dd}>> \mathcal P_{r}^{-}\Lambda^{n} @>{}>> 0. 
\end{CD}
$$

For each simplex $K \in\mathcal T_{h}$, denote $\mathcal P_{r}\Lambda^{k}(K)$ or $\mathcal P_{r}^{-}\Lambda^{k}(K)$ as the spaces of $k$ forms obtained by restricting the forms $\mathcal P_{r}\Lambda^{k}(\mathbb R^{n})$ or $\mathcal P_{r}^{-}\Lambda^{k}(\mathbb R^{n})$, respectively, to $K$. We then obtain the finite element spaces
\begin{align*}
\mathcal P_{r}\Lambda^{k}(\mathcal T_{h}) & = \{\omega \in H\Lambda^{k}(\Omega):~~\omega|_{K} \in \mathcal P_{r}\Lambda^{k}(K),~~\forall ~~K \in \mathcal T_{h}\}, \\
\mathcal P_{r}^{-}\Lambda^{k}(\mathcal T_{h}) & = \{\omega \in H\Lambda^{k}(\Omega):~~\omega|_{K} \in \mathcal P_{r}^{-}\Lambda^{k}(K),~~\forall ~~K \in \mathcal T_{h}\}.
\end{align*}
We choose $V_{h}^{k} = \mathcal P_{r}\Lambda^{k}(\mathcal T_{h}) \cap H_{0}\Lambda^{k}$ 
or $V_{h}^{k} = \mathcal P_{r}^{-}\Lambda^{k}(\mathcal T_{h}) \cap H_{0}\Lambda^{k}$ 
so that $(V_{h}^{k},\dd)$ forms a subcomplex of $(H_{0}\Lambda^{k}(\Omega),\dd)$. For the consecutive spaces, we shall use short sequence 
$$
\begin{CD}
V_{h}^{-} @>{\dd^{-}}>>  V_{h} @>{\dd}>> V_{h}^{+}. 
\end{CD}
$$
The discrete coderivative $\delta_{h}:~V_{h} \rightarrow V_{h}^{-}$ is defined as the $L^{2}$-adjoint of $\dd^{-}:~V_{h}^{-} \rightarrow V_{h}$, i.e., for any given $\omega_{h} \in V_{h}$, $\delta_{h}\omega_{h} $ is the unique element in $V_{h}^{-}$ such that
\begin{equation}\label{eq:delta_h}
\la\delta_{h}\omega_{h},v_{h}\ra  = \la \omega_{h},\dd^{-}v_{h}\ra\qquad \forall~~v_{h} \in V_{h}^{-}.
\end{equation}

The discrete Hodge decomposition of $V_{h}^{k}$ is 
\begin{equation}
\label{eq:hodge_dec_dis}
V_{h}^{k} = \mathfrak Z_{0,h} \oplus^{\bot_{L^{2}}} \mathfrak K_{h},
\end{equation}
where $\mathfrak Z_{0,h} = \ker(\dd) \cap V_{h} = \dd^{-}\bs V_{h}^{-} \subset \mathfrak Z_{0}$ and $\mathfrak K_{h} = \delta_{h}^{+} V_{h}^{+}$ is the $L^{2}$ orthogonal complement of $\mathfrak Z_{h,0}$ in $V_{h}$. Generally $\mathfrak K_{h} \not\subset \mathfrak K$, since $\delta_{h}$ is not a conforming discretization of $\delta$. It should be point out that when $k = 0$, we have $\mathfrak Z_{0,h} = \{0\}$ and $\mathfrak K_{h}^{-} = \{0\}$. When $k = n$, we have $\mathfrak K_{h} = \{0\}$. 

We have the following discrete Poincar\'e inequality; cf. \cite[Theorem 5.11]{Arnold;Falk;Winther2006}
\begin{lemma}[discrete Poincar\'e inequality for $\dd$]\label{lem:poincare-dis}
There is a positive constant $C_{p}$, independent of $h$, such that
\begin{equation}\label{eq:poincare-dis}
\|\omega_{h}\|  \leq C_{p}\|\dd\omega_{h}\|\qquad\forall ~~\omega_{h} \in \mathfrak K_{h}.
\end{equation}
\end{lemma}

Since $\delta_{h}$ is the adjoint operator of $\dd^{-}:~V_{h}^{-} \rightarrow V_{h}$, we have the following discrete Poincar\'e inequality for $\delta_{h}$ as well; cf. \cite{Chen2016MultiGrid} and \cite[Lemma 4]{Chenwu2017}.
\begin{lemma}[discrete Poincar\'e inequality for $\delta_{h}$]
Let $C_{p}$ be the constant in \eqref{eq:poincare-dis}. Then we have
$$
\|\omega_{h} \|  \leq C_{p}\|\delta_{h} \omega_{h}\|\qquad \forall~~\omega_{h} \in \mathfrak Z_{0,h}.
$$
\end{lemma}

\subsection{A projection-based quasi-interpolation operator}
In this section, we introduce a projection-based quasi-interpolation operator briefly mentioned in \cite[Proposition 5.44]{Christiansen2011} which is a generalization of projection based operators introduced for $H^1, H(\curl)$ and $H(\div)$ spaces in \cite{Demkowiz2000,Demkowiz2003,Demkowiz2005,Oden1989}. We redefine this operator based on the Hodge decomposition and prove more properties of this operator: it is commuted with $\delta_{h}$, stable in both $\|\dd(\cdot)\|$ and $\|\delta_{h}(\cdot)\|$ norms, a $L^{2}$ orthogonal projection to the space $\mathfrak Z_{0,h}$, an orthogonal projection operator in the inner product $\la \dd(\cdot),\dd(\cdot)\ra$ on the subspace $\mathfrak K_h$, and has the same approximation properties as the classical interpolation operators.

For any given $v\in H_{0}\Lambda(\Omega)$, define $P_{h} v \in \mathfrak K_{h}$ such that
\begin{equation}
\label{eq:p-h-def}
\la \dd P_{h} v,\dd \phi_{h}\ra = \la\dd v,\dd\phi_{h}\ra \qquad\forall~~\phi_{h} \in \mathfrak K_{h}.
\end{equation}
Equation \eqref{eq:p-h-def} determines $P_{h}v \in \mathfrak K_{h}$ uniquely since the Poincar\'e inequality \eqref{eq:poincare-dis} implies $\la\dd(\cdot),\dd(\cdot)\ra$ is an inner product on the subspace $\mathfrak K_{h}$. 

For any $v \in H_{0}\Lambda(\Omega)$, the Hodge decomposition \eqref{eq:hodge} implies that there exist $v_{1} \in \mathfrak K^{-}$ and $v_{2} \in \mathfrak K$ such that
$$
v = \dd^{-} v_{1} \oplus^{\bot_{L^{2}}} v_{2} 
$$
The projection-based quasi-interpolation operator $I_{h}:~H_{0}\Lambda(\Omega) \rightarrow V_{h}$ is defined as:
\begin{equation}
\label{eq:quasi-int}
I_{h} v = \dd^{-} P_{h}^{-} v_{1}   \oplus^{\bot_{L^{2}}} P_{h} v_{2}.
\end{equation}
We have the following properties.
\begin{lemma}\label{eq:I-h-orth}
For any $v\in H_{0}\Lambda(\Omega)$, there hold
$$
\la I_{h} v,\dd^{-} \phi_{h} \ra = \la v,\dd^{-} \phi_{h} \ra\qquad \forall ~~\phi_{h} \in V_{h}^{-},
$$
and
$$
\la \dd I_{h} v,\dd\psi_{h}\ra = \la \dd v,\dd\psi_{h} \ra \qquad\forall~~\psi_{h} \in V_{h}.
$$
Here we denote $V_{h}^{-1} = \{0\}$.
\end{lemma}
\begin{proof}
For any $\phi_{h} \in V_{h}^{-}$, the discrete Hodge decomposition \eqref{eq:hodge_dec_dis} implies that there exists $\phi_{h,1} \in \mathfrak K_{h}^{-}$ such that $\dd^{-} \phi_{h} = \dd^{-} \phi_{h,1}$, therefore 
$$
\la I_{h }v,\dd^{-} \phi_{h} \ra = \la\dd^{-} P_{h}^{-} v_{1},\dd^{-}\phi_{h,1} \ra = \la \dd^{-} v_{1},\dd^{-} \phi_{h,1}\ra = \la v,\dd^{-}\phi_{1,h}\ra = \la v,\dd^{-}\phi_{h}\ra.
$$

For any $\psi_{h} \in V_{h}$, the discrete Hodge decomposition \eqref{eq:hodge_dec_dis} implies that there exists $\psi_{h,1} \in \mathfrak K_{h}$ such that $\dd\psi_{h} = \dd\psi_{h,1}$, 
$$
\la \dd I_{h}v,\dd\psi_{h}\ra = \la \dd P_{h} v_{2},\dd\psi_{h,1}\ra = \la \dd v_{2},\dd \psi_{h,1}\ra = \la \dd v,\dd\psi_{h} \ra.
$$
Then, the desired results are obtained.
\end{proof}

We have the following stability results of $I_{h}$.
\begin{lemma}\label{lem:bd-I-h}
We have the following stability results of $I_{h}$:
\begin{enumerate}
  \item For any $v \in H_{0}\Lambda(\Omega)$, there holds
$$
 \|\dd I_{h}v\| \leq \|\dd v\|.
$$
  \item For any $v \in H_{0}\Lambda(\Omega) \cap H^{*}\Lambda(\Omega)$, it holds 
  $$
  \delta_{h} I_{h} v = Q_{h}^{-}\delta v,
  $$
  where $Q_{h}:~L^{2}\Lambda(\Omega) \rightarrow V_{h}$ is the $L^{2}$ projection operator. Therefore
  $$
  \|\delta_{h}I_{h}v\|  \leq \|\delta v\|.
  $$
\end{enumerate}
\end{lemma}
\begin{proof}
(1) For any $v \in H_{0}\Lambda(\Omega)$, there exist $v_{1} \in \mathfrak K^{-}$ and $v_{2} \in \mathfrak K$ such that
$$
v = \dd^{-} v_{1}  \oplus^{\bot_{L^{2}}} v_{2}.
$$
We have
\begin{align*}
\|\dd I_{h} v\| = \|\dd P_{h} v_{2}\| \leq \|\dd v_{2}\| = \|\dd v\|.
\end{align*}

(2) For any $v \in H_{0}\Lambda(\Omega) \cap H^{*}\Lambda(\Omega)$. Using the facts that $P_h v_2 \in \delta_h^+ V_h^+$, $\delta_h\delta_h^+ = 0$ and $\delta_{h}I_{h} v \in \mathfrak K_{h}^{-}$, for any $\phi_{h} \in \mathfrak K_{h}^{-}$, we have
\begin{align*}
\la \delta_{h} I_{h}v,\phi_{h} \ra & = \la \delta_{h} \dd^{-}P_{h}^{-}v_{1},\phi_{h} \ra = \la \dd^{-}P_{h}^{-}v_{1} ,\dd^{-} \phi_{h} \ra \\
& = \la \dd^{-} v_{1},\dd^{-}\phi_{h} \ra = \la v,\dd^{-}\phi_{h}\ra \\
& = \la \delta v,\phi_{h}\ra  = \la Q_{h}^{-}\delta v, \phi_{h}\ra.
\end{align*}
Using the orthogonality result $\mathfrak Z_{0,h}^{-}\bot\mathfrak K_{h}^{-}$, we get the desired result.
\end{proof}

To get approximation properties of the projection-based quasi-interpolation operator $I_h$, we need the de Rham complexes for smooth differential forms established in \cite{Costabel2010a} and the following Sobolev embedding result
\begin{equation}\label{eq:ass-omega}
H_0\Lambda \cap H^*\Lambda \hookrightarrow H^1\Lambda,
\end{equation}
which holds when $\Omega$ is convex Lipschitz domain. 

\begin{lemma}\label{lem:app-I-h}
Assume that $\Omega$ is smooth enough such that \eqref{eq:ass-omega} holds, then for any $v \in H_{0}\Lambda(\Omega) \cap H^{r+1}\Lambda(\Omega)$ with $r\geq 1$, we have
\begin{align}
\label{eq:app-1}
\|v - I_{h} v\| & \lesssim h^{l} \|v\|_{l}\qquad\text{for}\quad 1 \leq l \leq r, \\
\label{eq:app-2}
\|\dd(v - I_{h}v)\| & \lesssim h^{l}\|\dd v\|_{l}\qquad\text{for}\quad 1\leq l \leq r.
\end{align}
\end{lemma}
\begin{proof}
We shall use the de Rham sequence in \cite{Costabel2010a}. Then for any $v \in H_{0}\Lambda(\Omega) \cap H^{r}\Lambda(\Omega)$, there exist $v_{1} \in \mathfrak K^{-} \cap H^{r}\Lambda^{-}(\Omega)$ and $v_{2} \in \mathfrak K \cap H^{r}\Lambda(\Omega)$ such that
$$
v = \dd^{-} v_{1} \oplus^{\bot_{L^{2}}} v_{2},
$$
and 
\begin{equation}\label{eq:noremRE}
\|v_2\|_l \lesssim \|v\|_l.
\end{equation}

Therefore,
$$
I_{h}v = \dd^{-} P_{h}^{-} v_{1} \oplus^{\bot_{L^{2}}} P_{h}v_{2}.
$$

Note that $v_{2}$ is the solution of the problem
\begin{equation}\label{eq:aux-1}
v_{2} = \delta^{+} s,\quad \dd v_{2} = q\quad\text{in}\quad\Omega,\qquad \tr s = 0\quad\text{on}\quad\partial\Omega.
\end{equation}
with
$q = \dd v$. Then $P_{h} v_{2}$ is the mixed finite element approximation of $v_{2}$ in $V_{h}$, the standard error estimates of the mixed finite element method \cite{Brezzi;Fortin1991,Girault;Raviart2012} implies that
$$
\|v_{2} - P_{h}v_{2}\| \lesssim h^{l}\|v_{2}\|_{l} 
\lesssim h^{l}\|v\|_{l},
$$ 
where in the second inequality, we have used \eqref{eq:noremRE}. 

Also note that $v_{1}$ is the solution of the problem
\begin{equation}
\label{eq:aux-2}
\delta\dd^{-} v_{1} = g,\quad \delta^{-} v_{1} = 0,\quad \tr v_{1}  = 0\quad\text{on}\quad \partial\Omega,
\end{equation}
with $g = \delta v$. The definition of $P_{h}^{-}$ implies that $P_{h}^{-}v_{1}$ is the mixed finite element approximation of $v_{1}$ in $V_{h}^{-}$, then the standard error estimates for the mixed finite element methods \cite{Brezzi;Fortin1991,Girault;Raviart2012} imply
$$
\|\dd^{-}(v_{1} - P_{h}^{-}v_{1})\|  \lesssim h^{l} \|\dd^{-}v_{1}\|_{l} \lesssim h^{l}\|g\|_{l-1}  = h^{l} \|\delta v\|_{l-1} \lesssim \|v\|_{l}.
$$
Therefore,
$$
\|v - I_{h} v\| \leq \|\dd^{-}(v_{1} - P_{h}^{-}v_{1})\|  + \|v_{2} - P_{h}v_{2}\| \lesssim h^{l}\|v\|_{l}.
$$

We turn to the estimates of \eqref{eq:app-2}. Since for any $\phi_{h} \in \mathfrak K_{h}$, it holds
\begin{align*}
\la\dd P_{h} v_{2},\dd\phi_{h}\ra & = \la\dd v_{2},\dd \phi_{h} \ra = \la \dd \pi_{h}v_{2} + \dd(I - \pi_{h})v_{2},\dd\phi_{h}\ra,
\end{align*}
where $\pi_h:~H_0\Lambda(\Omega)\cap H^r\Lambda(\Omega)\rightarrow V_h$ is the classical interpolation operator \cite{Arnold;Falk;Winther2006}.
Therefore, 
$$
\dd P_{h} v_{2} =  \dd\pi_{h} v_{2} + Q_{\mathfrak Z_{0,h}^{+}}\dd(I - \pi_{h})v_{2},
$$
where $Q_{\mathfrak Z_{0,h}^{+}}:~L^{2}\Lambda^{+} \rightarrow \mathfrak Z_{0,h}^{+}$ is the $L^{2}$ orthogonal projection operator.
Then, we have
\begin{align*}
\|\dd(v - I_{h}v)\| = \|\dd(v_{2} - P_{h}v_{2})\|
\leq \|(I - \pi_{h}^{+})\dd v_{2}\|  \lesssim h^{l} \|\dd v\|_{l}.
\end{align*}

\end{proof}

\subsection{Semi-discretization and error analysis}

The semi-discrete formulation \cite{Quenneville-Belair2015,Arnold2018} of \eqref{eq:mix1}-\eqref{eq:mix3} is: Given $f \in L^{2}((0,T),L^{2}\Lambda)$, find $\bs u_{h} = (\sigma_{h},\mu_{h},\omega_{h})^{\intercal}:~(0,T]\mapsto  V_{h}^{-} \times V_{h} \times V_{h}^{+} := \bs W_{h}$ such that
\begin{align}
\label{eq:mix1-dis}
\langle \sigma_{h,t},\tau_{h}\rangle - \langle\dd^{-}\tau_{h},\mu_{h}\rangle  & = 0\,\qquad\qquad\forall~~\tau_{h} \in V_{h}^{-}, \\
\label{eq:mix2-dis}
\langle \mu_{h,t},v_{h}\rangle + \langle \dd^{-}\sigma_{h},v_{h}\rangle + \langle \omega_{h},\dd v_{h}\rangle  & = \langle f,v_{h}\rangle\qquad \forall~~v_{h}\in V_{h}, \\
\label{eq:mix3-dis}
\langle \omega_{h,t},\phi_{h}\rangle - \langle \dd \mu_{h},\phi_{h} \rangle & = 0\,\qquad\qquad \forall~~\phi_{h} \in V_{h}^{+},
\end{align}
with initial values
$$
\sigma_{h}(\cdot,0) = I_{h}^{-}\delta u_{0},\quad \mu_{h}(\cdot,0) = I_{h}u_{1},\quad \omega_{h}(\cdot,0) = I_{h}^{+}\dd u_{0}. 
$$
Introduce
$$
\mathcal A_{h} = \begin{pmatrix}
0  & \delta_{h}  & 0 \\ -\dd^{-}   & 0  & -\delta_{h}^{+} \\
0  & \dd & 0
\end{pmatrix}
$$
\eqref{eq:mix1-dis} - \eqref{eq:mix3-dis} can be rewritten as
\begin{equation}
\label{eq:mix-semi-dis} 
\la \bs u_{h,t},\bs v_{h}\ra + \la\mathcal A_{h} \bs u_{h},\bs v_{h}\ra = \la\bs F,\bs v_{h} \ra \qquad\forall~~\bs v_{h} \in \bs W_{h}.
\end{equation}

Following the same line as the proof of Theorem \ref{the:energy_continuous}, we have the energy estimates.
\begin{theorem}\label{the:energy_continuous-dis}
Let $\bs u_{h} = (\sigma_{h},\mu_{h},\omega_{h})^{\intercal}\in \bs W_{h}$ be the solution of the mixed formulation \eqref{eq:mix1-dis}-\eqref{eq:mix3-dis} or \eqref{eq:mix-semi-dis}. Provided $f \in L^{1}((0,T),L^{2}\Lambda)$, we have the energy bound
\begin{equation}\label{eq:energy1-dis}
\sup\limits_{0\leq s\leq T} \|\bs u_{h}(\cdot,t)\| \leq \|\bs u_{h}(\cdot,0)\| + 2\int_{0}^{T}\|f(\cdot,s)\| \dd s.
\end{equation}
Furthermore, if $f\in W^{1,1}((0,T),L^{2}\Lambda)$, we have the bound
\begin{equation}\label{eq:energy2-dis}
\sup\limits_{0\leq t \leq T} \|\mathcal A_{h}\bs u_{h}(\cdot,t)\| \leq \|\mathcal A_{h}\bs u_{h}(\cdot,0)\| + 4\|f\|_{L^\infty(L^2)} + 2\int_{0}^{T} \|f_{t}\|\dd t.
\end{equation}
When $f = 0$, these inequalities become equalities and we have the energy conservation 
$$
\|\bs u_h(\cdot,t)\| = \|\bs u_h(\cdot,0)\|, and  \ \|\mathcal A_h\bs u_h(\cdot,t)\| = \|\mathcal A_h\bs u_h(\cdot,0)\|\quad \forall t > 0.
$$
\end{theorem}
\begin{remark}
Since \eqref{eq:mix1-dis}-\eqref{eq:mix3-dis} (or \eqref{eq:mix-semi-dis}) is a linear system, Theorem \ref{the:energy_continuous-dis} implies the existence and uniqueness of the solution at any time level $t \in (0,T]$. Also, when the source term $f = 0$, Theorem \ref{the:energy_continuous-dis} implies that the energies $\|\bs u_{h}(\cdot,t)\|$ and $\|\mathcal A_{h}\bs u_{h}(\cdot,t)\|$ are preserved exactly.
\end{remark}

The rest of this section will focus on the error estimates of the semi-discretization \eqref{eq:mix1-dis}-\eqref{eq:mix3-dis} (or its simplified form \eqref{eq:mix-semi-dis}). We denote
$$
\mathcal I_{h} = \begin{pmatrix}
I_{h}^{-} & 0 & 0 \\
0 & I_{h} & 0 \\
0 & 0 & I_{h}^{+}
\end{pmatrix}.  
$$
Then for any $\bs v_{h} = (\tau_{h},v_{h},\phi_{h})^{\intercal} \in\bs W_{h}$, \eqref{eq:hodge-operator-form} is equivalent to
\begin{equation}\label{eq:hodge-inter-form}
\la\mathcal I_{h}\bs u_{t},\bs v_{h}\ra + \la \mathcal A_{h}\mathcal I_{h}\bs u,\bs v_{h}\ra  = \la\bs F,\bs v_{h}\ra + \la \Theta_{h,t},\bs v_{h} \ra + \la\mathcal A_{h}\mathcal I_{h}\bs u - \mathcal A\bs u,\bs v_{h}\ra 
\end{equation}
with
$$
\Theta_{h} = \mathcal I_{h}\bs u - \bs u.
$$
Using the properties of the projection-based quasi-interpolation operator $I_{h}$, we obtain
\begin{align*}
\la \mathcal A_{h} \mathcal I_{h} \bs u - \mathcal A \bs u,\bs v_{h} \ra & = \la I_{h}\mu - \mu,\dd^{-} \tau_{h} \ra - \la \dd^{-}(I_{h}^{-} - I)\sigma,v_{h}\ra \\
& \quad - \la I_{h}^{+} \omega - \omega,\dd v_{h} \ra + \la \dd (I_{h} - I) \mu,\phi_{h}\ra \\
& = -\la \dd^{-}(I_{h}^{-} - I)\sigma,v_{h}\ra + \la \dd (I_{h} - I)\mu,\phi_{h}\ra \\
& = \la \bs G,\bs V_{h}\ra,
\end{align*}
with $\bs G = (0,-\dd^{-}(\pi_{h}^{-} - I) \sigma,\dd (I_{h} - I)\mu)^{\intercal}$. Denote 
$$
\mathcal E_{h} = \mathcal I_{h}\bs u - \bs u_{h}
$$
and subtracting the semi-discrete form \eqref{eq:mix-semi-dis} from \eqref{eq:hodge-inter-form}, we get 
\begin{align}
\label{eq:err-semi}
\la\mathcal E_{h,t},\bs v_{h}\ra + \la\mathcal A_{h}\mathcal E_{h},\bs v_{h}\ra  = \la\Theta_{h,t} + \bs G,\bs v_{h}\ra\qquad\forall~~\bs v_{h} \in \bs W_{h}.
\end{align}

We have the following estimate of $\mathcal E_{h}$.
\begin{lemma}\label{lem:err-L-2-dis}
Suppose the exact solution $\bs u = (\sigma,\mu,\omega)$ of \eqref{eq:hodge-operator-form} has time derivatives $\sigma_{t} \in L^{1}((0,T),H^{r}\Lambda^{-})$, $\mu_{t} \in L^{1}((0,T),H^{r}\Lambda)$ and $\omega_{t} \in L^{1}((0,T),H^{r}\Lambda^{+})$ with $r \geq 1$. Then, for any $1\leq m \leq r$ and $t \in [0,T]$, we have the bound
$$
\|\mathcal E_{h}(\cdot,t)\| \lesssim h^{m} \int_{0}^{T} \left( \|\bs u_{t}\|_{m} + \|\dd^{-}\sigma\|_{m} + \|\dd\mu\|_{m}  \right) \dd t.
$$
\end{lemma}
\begin{proof}
The fact that $\mathcal E_{h}(\cdot,0) =0$ and Theorem \ref{the:energy_continuous-dis} implies 
$$
\sup\limits_{0\leq t \leq T}\|\mathcal E_{h}(\cdot,t)\| \leq 2\int_{0}^{T}\|\Theta_{h,t} + \bs G\|\dd t.
$$
Using the triangle inequality and the approximation properties of $I_{h}$, the desired result follows.
\end{proof}

We then obtain the following estimates by the triangle inequality, Lemma \ref{lem:app-I-h}, and \ref{the:energy_continuous-dis}.
\begin{theorem}\label{the:err-semi-L2}
Suppose the exact solution $\bs u = (\sigma,\mu,\omega)^{\intercal}$ of \eqref{eq:hodge-operator-form} has time derivatives $\sigma_{t} \in L^{1}((0,T),H^{r}\Lambda^{-})$, $\mu_{t} \in L^{1}((0,T),H^{r}\Lambda)$ and $\omega_{t} \in L^{1}((0,T),H^{r}\Lambda^{+})$ with $r \geq 1$. 
Let $\bs u_h $ be the exact solution of \eqref{eq:mix-semi-dis}.
Then, for any $1\leq m \leq r$ and $t \in [0,T]$, we have the bound
\begin{align*}
\|\bs u(\cdot,t) - \bs u_{h}(\cdot,t)\| & \lesssim h^{m}\left( \|\bs u\|_{L^{\infty}(H^{m})} +\int_{0}^{T} (\|\bs u_{t}\|_{m} + \|\dd^{-}\sigma\|_{m} + \|\dd\mu\|_{m})\dd t  \right) .
\end{align*}
\end{theorem}

\begin{remark}
In this theorem, the convergence order $m$ is determined by the polynomial order of the finite element spaces preserved. 
\end{remark}
\begin{remark}
It should be point out that in \cite{Quenneville-Belair2015}, Quenneville has obtained an error estimates for the semi-discretization in the form
\begin{align*}
\|\bs u - \bs u_h\|_{L^\infty(L^2)} & \leq \|\pi_h\bs u_0 - \bs u_{0,h}\| + \|\pi_h\bs u - \bs u\|_{L^\infty(L^2)} \\
&\quad + (1+T)(\|(\pi_h\bs u - \bs u)(\cdot,0)\| + \|\pi_h\bs u_t - \bs u_t\|_{L^1(L^2)}),
\end{align*}
where $\pi_h$ is an elliptic projection operator.
Comparing with this result, ours do not have the factor $1+T$ and thus is more robust to the time variable. $\qed$
\end{remark}
	

We now give to the error estimates in the energy norm $\|\mathcal A \bs u - \mathcal A_h\bs U_h\|$, which is equivalent to $\|d^-\sigma - d^-\sigma_h \| + \|\dd\mu - \dd \mu_h\| + \|\delta\mu - \delta_h\mu\| + \|\delta^+\omega - \delta_h^+\omega_h\|$. Note that it is possible that the $L^2$-norm is small but the energy norm is larger due to the  small oscillation in the error. We shall show the energy norm is still of the same order of convergence. We give the estimate of $\|\mathcal A_{h}\mathcal E_{h}\|$ first.
By Lemma \ref{lem:app-I-h} and \ref{the:energy_continuous-dis}, we have the following estimate.
\begin{lemma}\label{lem:err-H-h}
Suppose the exact solution $\bs u = (\sigma,\mu,\omega)^{\intercal}$ of \eqref{eq:hodge-operator-form} has time derivatives $\sigma_{tt} \in L^{1}((0,T),H^{r}\Lambda^{-})$, $\mu_{tt} \in L^{1}((0,T),H^{r}\Lambda)$ and $\omega_{tt} \in L^{1}((0,T),H^{r}\Lambda^{+})$ with $r \geq 1$. Then, for any $1\leq m \leq r$ and $t \in [0,T]$, we have the bound
\begin{align*}
\|\mathcal A_{h}\mathcal E_{h}(\cdot,t)\| &\lesssim h^m(\|\bs u_t\|_{L^\infty(H^m)} + \|\dd^-\sigma\|_{L^\infty(H^m)} + \|\dd\mu\|_{L^\infty(H^m)} )\\
&\quad  + h^{m}\int_{0}^{T} \left( \|\bs u_{tt}\|_{m} + \|\dd^{-} \sigma_{t}\|_{m} + \|\dd \mu_{t}\|_{m}  \right) \dd t
\end{align*}
\end{lemma}
\begin{proof}
Using the fact that $\mathcal E_{h}(\cdot,0) = 0$ and Theorem \ref{the:energy_continuous-dis}, we have
$$
\|\mathcal A_{h}\mathcal E_{h}(\cdot,t)\| \leq 4 \|\Theta_{h,t} + \bs G\|_{L^\infty(L^2)} +  2\int_{0}^{T}\|\Theta_{h,tt} + \bs G_{t}\|\dd t.
$$
Triangle inequality and Lemma \ref{lem:app-I-h} imply the desired result.
\end{proof}

\begin{theorem}\label{the:err-semi-A}
Suppose the exact solution $\bs u = (\sigma,\mu,\omega)^{\intercal}$ of \eqref{eq:hodge-operator-form} has time derivatives $\sigma_{tt} \in L^{1}((0,T),H^{r}\Lambda^{-})$, $\mu_{tt} \in L^{1}((0,T),H^{r}\Lambda)$ and $\omega_{tt} \in L^{1}((0,T),H^{r}\Lambda^{+})$ with $r \geq 1$. Then, for any $1\leq m \leq r$ and $t \in [0,T]$, we have the bound
\begin{align*}
\|\mathcal A_{h}\bs u_{h}(\cdot,t) - \mathcal A\bs u(\cdot,t)\|  \lesssim & h^m {\large [} \, \|\bs u_t\|_{L^\infty(H^m)} + \|\dd^-\sigma\|_{L^\infty(H^m)} + \|\dd\mu\|_{L^\infty(H^m)} \\
& \|\mathcal A \bs u\|_{L^{\infty}(H^{m})} + \int_{0}^{T}(\|\bs u_{tt}\|_{m} + \|\dd^{-}\sigma_{t}\|_{m} + \|\dd\mu_{t}\|) \dd t \, {\large ]}.
\end{align*}
\end{theorem}
\begin{proof}
The triangle inequality and Lemma \ref{lem:bd-I-h} imply that
\begin{align*}
\|\mathcal A_{h}\bs u_{h}(\cdot,t) - \mathcal A\bs u(\cdot,t)\| &  \lesssim \|\dd^{-}(\sigma - I_{h}^{-}\sigma)\| + \|\dd(\mu - I_{h}\mu)\|  + \|\delta\mu - \delta_{h}I_{h}\mu\| \\
&\quad + \|\delta^{+} \omega- \delta^{+}_{h} I_{h}^{+}\omega\| + \|\mathcal A_{h}\mathcal E_{h}(\cdot,t)\| \\
& = \|\dd^{-}(I - I_{h}^{-})\sigma\| + \|\dd(\mu - I_{h}\mu)\| + \|(I - Q_{h}^{-})\delta\mu\|\\
& \quad + \|(I - Q_{h})\delta^{+}\omega\| + \|\mathcal A_{h}\mathcal E_{h}(\cdot,t)\|.
\end{align*}
Using the properties of the $L^{2}$ projection operators, Lemma \ref{lem:app-I-h} and \ref{lem:err-H-h}, we get the desired results.
\end{proof}

\section{Full-discretization}
In this section, we will consider the full discretization. We will use a second order continuous time Galerkin method \cite{French1996A} to discretize time variable and will obtain the energy estimates and optimal error estimates.

Energy conservation numerical schemes can have a crucial influence on the quality of the numerical simulations. In long-time simulations, energy-preserving can have a dramatic effect on stability and global error growth. The numerical schemes are not automatically inherit from the semi-discretization and a lot of time discretization methods cannot preserve the energies exactly. These led us to pay more attentions on the time discretization.


\subsection{Time discretization}
Let $\mathcal T_{\dt}$ denote the equispaced partition of the interval $(0,T)$ with $\dt = T/N$ and $N$ the number of elements in $\mathcal T_{\dt}$. For $1\leq i \leq N$, we denote $t_{i} = i\dt$ and $\tau_{i} = (t_{i-1},t_{i})$ with $t_{0} = 0$. For any quantity $v(t)$, we denote $v^{i} = v(t_{i})$. Define $\mathcal P_{1}(\mathcal T_{\dt})$ (abbr. $\mathcal P_{1}$) as the set of continuous piecewise linear polynomials with respect to the time variable $t$ on $\mathcal T_{\dt}$ and $\mathcal P_{0}(\mathcal T_{\dt})$ (abbr. $\mathcal P_{0}$) as the set of piecewise constant with respect to the time variable $t$ on $\mathcal T_{\dt}$. For any Sobolev space $S$ associates with the spatial variables, we use
$\mathcal P_{1}(S)$ to denote the set of functions that are continuous piecewise linear polynomials with respect to the time variable $t$ and in the Sobolev space $S$ with respect to the spatial variables. $\mathcal P_{0}(S)$ is defined similarly. 

The full discrete formulation of the Hodge wave equation \eqref{eq:hodge-operator-form} can be written as: Find $\bs U_{h} = (\tilde\sigma_{h},\tilde\mu_{h},\tilde\omega_{h})^{\intercal}   \in \mathcal P_{1}(  \bs W_{h})$ such that
\begin{equation}
\label{eq:hodge-wave-full-dis}
\int_{0}^{T} \left(\la \bs U_{h,t},\bs V_{h}\ra + \la \mathcal A_{h} \bs U_{h},\bs V_{h}\ra\right) \dd t = \int_{0}^{T} \la \bs F,\bs V_{h}\ra \dd t\qquad\forall~~~\bs V_{h} \in  \mathcal P_{0}( \bs W_{h}).
\end{equation}
\begin{remark}
The full discrete formulation \eqref{eq:hodge-wave-full-dis} is equivalent to
$$
\int_{t_{i-1}}^{t_{i}}(\la\bs U_{h,t},\bs V_{h} \ra  + \la\mathcal A_{h} \bs U_{h},\bs V_{h}\ra )\dd t = \int_{t_{i-1}}^{t_{i}}\la \bs F,\bs V_{h}\ra \dd t\quad\forall~\bs V_{h}\in \mathcal P_{0}(\bs W_{h}),\  1\leq i \leq N.
$$
The fact that
\begin{align*}
\int_{t_{i-1}}^{t_{i}}\la\bs U_{h,t},\bs V_{h} \ra \dd t  & = \la\bs U_{h}^{i} - \bs U_{h}^{i-1},\bs V_{h}\ra, 
\end{align*}
and
$$
\int_{t_{i-1}}^{t_{i}}\la\mathcal A_{h} \bs U_{h},\bs V_{h}\ra \dd t  = \frac{\dt}{2}\la\mathcal A_{h}(\bs U_{h}^{i} + \bs U_{h}^{i-1}),\bs V_{h}\ra,
$$
implies the full discrete formulation \eqref{eq:hodge-wave-full-dis} is essentially a Crank-Nicolson scheme with exact time integration of the right hand side.
\end{remark}

We have the following energy estimates for \eqref{eq:hodge-wave-full-dis}.
\begin{theorem}\label{them:energy-l2-dis}
Let $\bs U_{h} = (\tilde\sigma_{h},\tilde \mu_{h},\tilde\omega_{h})^{\intercal} \in\mathcal P_{1} (\bs W_{h})$ be the solutions of \eqref{eq:hodge-wave-full-dis}. Assume that $f \in L^{\infty}((0,T), L^{2}\Lambda)$, then there hold the following energy bound
\begin{equation}
\label{eq:ener-dis-f}
\max\limits_{0\leq i \leq N} \|\bs U_{h}^{i}\| \leq \|\bs U_{h}^{0}\| + 2\int_{0}^{T}\|\bs F\|\dd t.
\end{equation}
When $F = 0$, the inequality becomes equality and we have the energy conservation 
$$\|\bs U_h^i\| = \|\bs U_h^0\|, \quad \forall \ 1\leq i \leq N.$$
\end{theorem}
\begin{proof}
Taking $\bs V_{h}$ in \eqref{eq:hodge-wave-full-dis} as $$
\bs V_{h}|_{\tau_{i}} = \bs U_{h}^{i} + \bs U_{h}^{i-1}\qquad\text{and}\qquad \bs V_{h}|_{\mathcal T_{\dt}\setminus \tau_{i}} = 0,
$$
we obtain
$$
\int_{t_{i-1}}^{t_{i}}\la \bs U_{h,t},\bs U_{h}^{i} + \bs U_{h}^{i-1}\ra \dd t + \int_{t_{i-1}}^{t_{i}}\la \mathcal A_{h}\bs U_{h},\bs U_{h}^{i} + \bs U_{h}^{i-1}\ra\dd t = \int_{t_{i-1}}^{t_{i}} \la \bs F,\bs U_{h}^{i} + \bs U_{h}^{i-1}\ra \dd t.
$$
The fact that
\begin{align*}
\int_{t_{i-1}}^{t_{i}} \la \mathcal A_{h}\bs U_{h},\bs U_{h}^{i} + \bs U_{h}^{i-1}\ra\dd t & = \frac{\dt}{2} \la\mathcal A_{h}(\bs U_{h}^{i} + \bs U_{h}^{i-1}),\bs U_{h}^{i} + \bs U_{h}^{i-1}\ra = 0
\end{align*}
and
\begin{align*}
\int_{t_{i-1}}^{t_{i}} \la \bs U_{h,t},\bs U_{h}^{i} + \bs U_{h}^{i-1}\ra \dd t & = \|\bs U_{h}^{i}\|^{2} - \|\bs U_{h}^{i-1}\|^{2}
\end{align*}
imply 
\begin{align*}
\|\bs U_{h}^{i}\|^{2} - \|\bs U_{h}^{i-1}\|^{2} & = \int_{t_{i-1}}^{t_{i}}\la\bs F,\bs U_{h}^{i} + \bs U_{h}^{i-1}\ra \dd t \leq 2 \max\limits_{0\leq i\leq N}\|\bs U_{h}^{i}\| \int_{t_{i-1}}^{t_{i}} \|\bs F\|\dd t.
\end{align*}
Summing over $i$ from $1$ to $m \leq N$, we get
\begin{align*}
\|\bs U_{h}^{m}\|^{2} - \|\bs U_{h}^{0}\|^{2} \leq 2\max\limits_{0\leq i \leq N}\|\bs U_{h}^{i}\| \int_{0}^{T}\|\bs F\|\dd t.
\end{align*}
Therefore, the desired result follows by Lemma \ref{lem:greatwall}.
\end{proof}

\begin{remark}
Since \eqref{eq:hodge-wave-full-dis} is a linear system, therefore Theorem \ref{them:energy-l2-dis} implies the existence and uniqueness of the solution for the full discrete form \eqref{eq:hodge-wave-full-dis}.
\end{remark}

\begin{theorem}
\label{the:ener-dis-f-2}
Let $\bs U_{h} = (\tilde\sigma_{h}, \tilde\mu_{h} ,\tilde\omega_{h})^{\intercal} \in\mathcal P_{1}(\bs  W_{h})$ be the solution of \eqref{eq:hodge-wave-full-dis}. Assume that $f \in W^{1,1}((0,T), L^{2}\Lambda)$, then there holds the following energy bound
\begin{equation}
\label{eq:ener-dis-f-2}
\max\limits_{0\leq i \leq N} \|\mathcal A_{h}\bs U_{h}^{i}\| \leq \|\mathcal A_{h}\bs U_{h}^{0}\| + 4 \|\bs F\|_{L^{\infty}(L^{2})} + 2\int_{0}^{T}\|\bs F_{t}\|\dd t.
\end{equation}
When $F = 0$, the inequality becomes equality and we have the energy conservation 
$$\|\mathcal A_h\bs U_h^i\| = \|\mathcal A_h \bs U_h^0\|\quad \forall \ 1\leq i \leq N.$$
\end{theorem}
\begin{proof}
Taking $\bs V_{h}$ in \eqref{eq:hodge-wave-full-dis} as
$$
\bs V_{h}|_{\tau_{i}} = \mathcal A_{h}(\bs U_{h}^{i} - \bs U_{h}^{i-1})\quad\text{and}\quad \bs V_{h}|_{\mathcal T_{\dt}\setminus \tau_{i}} = 0,
$$
we have
\begin{align*}
\int_{t_{i-1}}^{t_{i}} \la \bs U_{h,t},\mathcal A_{h}(\bs U_{h}^{i} - \bs U_{h}^{i-1})\ra \dd t & + \int_{t_{i-1}}^{t_{i}} \la \mathcal A_{h}\bs U_{h},\mathcal A_{h}(\bs U_{h}^{i} - \bs U_{h}^{i-1})\ra \dd t\\
& =   \int_{t_{i-1}}^{t_{i}} \la \bs F,\mathcal A_{h}(\bs U_{h}^{i} - \bs U_{h}^{i-1})\ra \dd t.
\end{align*}
Using the fact that
$$
\int_{t_{i-1}}^{t_{i}}\la \bs U_{h,t}, \mathcal A_{h}(\bs U_{h}^{i} - \bs U_{h}^{i-1})\ra \dd t = \la \bs U_{h}^{i} - \bs U_{h}^{i-1},\mathcal A_{h}(\bs U_{h}^{i} - \bs U_{h}^{i-1}) \ra = 0
$$
and
$$
\int_{t_{i-1}}^{t_{i}}\la\mathcal A_{h}\bs U_{h},\mathcal A_{h}(\bs U_{h}^{i} - \bs U_{h}^{i-1})\ra = \frac{\dt}{2}( \|\mathcal A_{h}\bs U_{h}^{i}\|^{2} - \|\mathcal A_{h}\bs U_{h}^{i-1}\|^{2}),
$$
we get
\begin{align*}
\|\mathcal A_{h}\bs U_{h}^{i}\|^{2} - \|\mathcal A_{h}\bs U_{h}^{i-1}\|^{2} & = 2 \int_{t_{i-1}}^{t_{i}}\left\la\bs F,\mathcal A_{h}\frac{\bs U_{h}^{i} - \bs U_{h}^{i-1}}{\dt} \right\ra \dd t 
 = 2\int_{t_{i-1}}^{t_{i}}\la\bs F,\mathcal A_{h}\bs U_{h,t}\ra \dd t.
\end{align*}
Summing over $i$ from $1$ to $m \leq N$, we obtain
\begin{align*}
\|\mathcal A_{h}\bs U_{h}^{m}\|^{2} &  = \|\mathcal A_{h} \bs U_{h}^{0}\|^{2} + 2 \int_{0}^{t_{m}}\la\bs F,\mathcal A_{h} \bs U_{h,t} \ra \dd t \\
& = \|\mathcal A_{h} \bs U_{h}^{0}\|^{2} + 2\la \bs F^{m},\mathcal A_{h}\bs U_{h}^{m}\ra - 2\la \bs F^{0},\mathcal A_{h}\bs U_{h}^{0}\ra - 2 \int_{0}^{t_{m}}\la \bs F_{t},\mathcal A_{h}\bs U_{h}\ra \dd t \\
& \leq \|\mathcal A_{h}\bs U_{h}^{0}\|^{2} + 2\max\limits_{0\leq i\leq N}\|\mathcal A_{h}\bs U_{h}^{i}\|\left(2 \|\bs F\|_{L^{\infty}(L^{2})} +  \int_{0}^{T}\|\bs F_{t}\|\dd t\right).
\end{align*}
Then the desired result follows by a direct using of Lemma \ref{lem:greatwall}.
\end{proof}

\subsection{Error analysis of the full discretization}

In this subsection, we turn to the error estimates of the full discrete formulation \eqref{eq:hodge-wave-full-dis}. We bound the error of the full discrete formulation in various norms.
Let 
\begin{align*}
\bs e_{h} = \mathcal I_{h}\bs u - \bs U_{h}.
\end{align*}
Simple caculation shows that for any $\bs V_{h} \in \mathcal P_{0} ( \bs W_{h})$, $\bs e_{h}$ satisfies the following equation
\begin{equation}
\label{eq:int-err-full-new}
\int_{0}^{T} (\la \bs e_{h,t},\bs V_{h}\ra  + \la\mathcal A_{h}\bs e_{h},\bs V_{h}\ra) \dd t = \int_{0}^{T}\la \Theta_{h} + \bs G,\bs V_{h} \ra \dd t,\quad\forall ~~\bs V_{h} \in \mathcal P_{0}(\bs W_{h}).
\end{equation}
Then, we have the following estimates of $\bs e_{h}$.
\begin{lemma}
\label{lem:e-l2-es}
Let $\bs u = (\sigma,\mu,\omega)^{\intercal}$ be the solution of \eqref{eq:hodge-operator-form} and $\bs U_{h} = (\tilde\sigma_{h},\tilde\mu_{h},\tilde\omega_{h})^{\intercal} \in \mathcal P_{1}(\bs W_{h})$ be the solution of \eqref{eq:hodge-wave-full-dis}. 
Assume that $\sigma_{tt} \in L^{1}((0,T),H^{r}\Lambda^{-})$, $\mu_{tt} \in L^{1}((0,T),H^{r}\Lambda)$ and $\omega_{tt} \in L^{1}((0,T),H^{r}\Lambda^{+})$ with $r \geq 1$. Then, for any $1\leq m \leq r$, we have the bound
\begin{align*}
\max\limits_{0\leq i\leq N} \|\bs e_{h}^{i}\|  \lesssim &h^{m} \int_{0}^{T}(\|\bs u_{h}\|_{m} + \|\dd^{-}\sigma\|_{m} + \|\dd\mu\|_{m}) \dd t + \dt^{2}\int_{0}^{T} \|\mathcal A\bs u_{tt}\| \dd t.
\end{align*}
\end{lemma}
\begin{proof}
Taking $\bs V_{h}$ in \eqref{eq:int-err-full-new} as
$$
\bs V_{h}|_{\tau_{i}} = \bs e_{h}^{i-1} + \bs e_{h}^{i}\quad\text{and}\quad \bs V_{h}|_{\mathcal T_{\dt}\setminus \tau_{i}} = 0,
$$
we obtain
\begin{align*}
\int_{t_{i-1}}^{t_{i}}\la \bs e_{h,t},\bs e_{h}^{i-1} + \bs e_{h}^{i}\ra \dd t  & + \int_{t_{i-1}}^{t_{i}} \la \mathcal A_{h}\bs e_{h},\bs e_{h}^{i} + \bs e_{h}^{i+1} \ra \dd t = \int_{t_{i-1}}^{t_{i}} \la \Theta_{h} + \bs G,\bs e_{h}^{i-1} + \bs e_{h}^{i} \ra \dd t.
\end{align*}
Note that
\begin{align*}
\int_{t_{i-1}}^{t_{i}}\la \bs e_{h,t},\bs e_{h}^{i-1} + \bs e_{h}^{i} \ra \dd t & = \|\bs e_{h}^{i}\|^{2} - \|\bs e_{h}^{i-1}\|^{2},
\end{align*}
and
\begin{align*}
& \int_{t_{i-1}}^{t_{i}} \la \mathcal A_{h}\bs e_{h},\bs e_{h}^{i-1} + \bs e_{h}^{i} \ra \dd t  = \int_{t_{i-1}}^{ t_{i}} \la \mathcal A_{h}(I - J_{\dt})\bs e_{h},\bs e_{h}^{i-1} + \bs e_{h}^{i} \ra \dd t \\
 \leq &2 \max\limits_{0\leq i \leq N}\|\bs e_{h}^{i}\| \int_{t_{i-1}}^{t_{i}} \|\mathcal A_{h}(I - J_{\dt})\bs e_{h}\|\dd t 
\leq C\dt^{2} \max\limits_{0\leq i leq N}\|\bs e_{h}^{i}\| \int_{t_{i-1}}^{t_{i}}\|\mathcal A_{h}\mathcal I_{h} \bs u_{tt}\| \dd t \\
 \leq & C \dt^{2} \max\limits_{0\leq i leq N}\|\bs e_{h}^{i}\| \int_{t_{i-1}}^{t_{i}}\|\mathcal A \bs u_{tt}\| \dd t ,
\end{align*}
we then have
\begin{align*}
&\|\bs e_{h}^{i}\|^{2}  - \|\bs e_{h}^{i-1}\|^{2}  \leq C \max\limits_{0\leq i \leq N}\|\bs e_{h}^{i}\| \left(\int_{t_{i-1}}^{t_{i}}(\|\Theta_{h}\| + \|\bs G\|) \dd t + \dt^{2}\int_{t_{i-1}}^{t_{i}} \|\mathcal A\bs u_{tt}\| \dd t \right) \\
\lesssim & \max\limits_{0\leq i \leq N}\|\bs e_{h}^{i}\| \left(h^{m} \int_{t_{i-1}}^{t_{i}}(\|\bs u_{h}\|_{m} + \|\dd^{-}\sigma\|_{m} + \|\dd\mu\|_{m}) \dd t + \dt^{2}\int_{t_{i-1}}^{t_{i}} \|\mathcal A\bs u_{tt}\| \dd t \right). 
\end{align*}
Summing over $i$ from $1$ to $m \leq N$, we get
\begin{align*}
\|\bs e_{h}^{m}\|^{2} & \lesssim  \max\limits_{0\leq i \leq N}\|\bs e_{h}^{i}\| \left(h^{m} \int_{0}^{T}(\|\bs u\|_{m} + \|\dd^{-}\sigma\|_{m} + \|\dd\mu\|_{m}) \dd t + \dt^{2}\int_{0}^{T} \|\mathcal A\bs u_{tt}\| \dd t \right).
\end{align*}
Then the desired result follows. 
\end{proof}

Using triangle inequality, Lemma \ref{lem:e-l2-es} and the properties of $I_{h}$, we have the following $L^{2}$ error estimate.
\begin{theorem}\label{the:err-full-L2}
Let $\bs u = (\sigma,\mu,\omega)^{\intercal}$ be the solution of \eqref{eq:hodge-operator-form} and $\bs U_{h} = (\tilde\sigma_{h},\tilde\mu_{h},\tilde\omega_{h})^{\intercal} \in \mathcal P_{1}(\bs W_{h})$ be the solution of \eqref{eq:hodge-wave-full-dis}.
Assume that $\sigma_{tt} \in L^{1}((0,T),H^{r}\Lambda^{-})$, $\mu_{tt} \in L^{1}((0,T),H^{r}\Lambda)$ and $\omega_{tt} \in L^{1}((0,T),H^{r}\Lambda^{+})$ with $r \geq 1$. Then, for any $1\leq m \leq r$ we have the bound
\begin{align*}
\max\limits_{0\leq i\leq N} \|\bs u^{i} - \bs U_{h}^{i}\|  \lesssim &
h^{m} \left(\|\bs u\|_{L^{\infty}(H^{m})} + \|\bs u\|_{L^{1}(H^{m})} + \|\dd^{-}\sigma\|_{L^{1}(H^{m})} + \|\dd\mu\|_{L^{1}(H^{m})} \right) \\
& + \dt^{2}\|\mathcal A\bs u_{tt}\|_{L^{1}(L^{2})}
%
\end{align*}
\end{theorem}

Now, we turn to the estimates of the energy norm $\|\mathcal A(\cdot)\|$ error. We first estimate $\|\mathcal A_{h}\bs e_{h}\|$.
\begin{lemma}
Let $\bs u = (\sigma,\mu,\omega)^{\intercal}$ be the solution of \eqref{eq:hodge-operator-form} and $\bs U_{h} = (\tilde\sigma_{h},\tilde\mu_{h},\tilde\omega_{h})^{\intercal} \in \mathcal P_{1}(\bs W_{h})$ be the solution of \eqref{eq:hodge-wave-full-dis}.
Assume that $\sigma_{tt} \in L^{\infty}((0,T),H^{r}\Lambda^{-})$, $\mu_{tt} \in L^{\infty}((0,T),H^{r}\Lambda)$ and $\omega_{tt} \in L^{\infty}((0,T),H^{r}\Lambda^{+})$ with $r \geq 1$. Then, for any $1\leq m \leq r$ we have the bound
\begin{align*}
\|\mathcal A_{h}\bs e_{h}\| & \lesssim h^{m}\left( \|\bs u_{t}\|_{L^{1}(H^{m})} + \|\dd^{-} \sigma_{t}\|_{L^{1}(H^{m})} + \|\dd\mu_{t}\|_{L^{1}(H^{m})} \right) + \dt^{2}\|\mathcal A \bs u_{tt}\|_{L^{\infty}(L^{2})}.
\end{align*}
\end{lemma}
\begin{proof}
Taking $\bs V_{h}$ in \eqref{eq:int-err-full-new} as
$$
\bs V_{h}|_{\tau_{i}} = \mathcal A_{h} (\bs e_{h}^{i} - \bs e_{h}^{i-1})\quad\text{and}\quad \bs V_{h}|_{\mathcal T_{\dt}\setminus \tau_{i}} = 0,
$$
we obtain
\begin{align*}
\int_{t_{i-1}}^{t_{i}} \la \bs e_{h,t},\mathcal A_{h}(\bs e_{h}^{i} - \bs e_{h}^{i-1})\ra \dd t & + \int_{t_{i-1}}^{t_{i}} \la \mathcal A_{h}\bs  e_{h},\mathcal A_{h}(\bs e_{h}^{i} - \bs e_{h}^{i-1}) \ra \dd t \\
&  = \int_{t_{i-1}}^{t_{i}}\la\Theta_{h}+\bs G,\mathcal A_{h}(\bs e_{h}^{i} -\bs e_{h}^{i-1}) \ra \dd t.
\end{align*}
Note that
\begin{align*}
\int_{t_{i-1}}^{t_{i}} \la\bs e_{h,t},\mathcal A_{h}(\bs e_{h}^{i} - \bs e_{h}^{i-1}) \ra \dd t & = \la(\bs e_{h}^{i} - \bs e_{h}^{i-1}),\mathcal A_{h}(\bs e_{h}^{i} - \bs e_{h}^{i-1})\ra = 0,
\end{align*}
and
\begin{align*}
\int_{t_{i-1}}^{t_{i}}\la\mathcal A_{h}\bs e_{h},\mathcal A_{h}(\bs e_{h}^{i} - \bs e_{h}^{i-1}) \ra \dd t & =\frac{\dt}{2} \left( \|\mathcal A_{h}\bs e_{h}^{i}\|^{2} - \|\mathcal A_{h}\bs e_{h}^{i-1}\|^{2} \right) \\
&\quad  + 
\int_{t_{i-1}}^{t_{i}}\la\mathcal A_{h} (I - J_{\dt})\bs e_{h},\mathcal A_{h}(\bs e_{h}^{i} - \bs e_{h}^{i-1})\ra \dd t.
\end{align*}
Therefore,
\begin{align*}
& \|\mathcal A_{h}\bs e_{h}^{i}\|^{2} - \|\mathcal A_{h}\bs e_{h}^{i-1}\|^{2} \\  = & \frac{2}{\dt} \int_{t_{i-1}}^{t_{i}} \la\Theta_{h}+\bs G,\mathcal A_{h}(\bs e_{h}^{i} - \bs e_{h}^{i-1})\ra \dd t  - \frac{2}{\dt} \int_{t_{i-1}}^{t_{i}} \la\mathcal A_{h} (I - J_{\dt})\bs e_{h},\mathcal A_{h}(\bs e_{h}^{i} - \bs e_{h}^{i-1})\ra \dd t\\
=& 2 \int_{t_{i-1}}^{t_{i}}\left\la \Theta_{h} + \bs G,\mathcal A_{h}\frac{\partial}{\partial t}(J_{\dt}\bs e_{h})\right\ra\dd t -2 \int_{t_{i-1}}^{t_{i}}\left\la \mathcal A_{h}(I - J_{\dt})\bs e_{h},\mathcal A_{h}\frac{\partial}{\partial t}(J_{\dt}\bs e_{h})\right\ra\dd t.
\end{align*}
Summing over $i$ from $1$ to $ m \leq N$, we get
\begin{align*}
& \|\mathcal A_{h}\bs e_{h}^{m}\|^{2} \\
  = &2\int_{0}^{t_{m}} \left\la \Theta_{h} + \bs G,\mathcal A_{h}\frac{\partial}{\partial t}(J_{\dt}\bs e_{h})\right\ra\dd t 
- 2 \int_{0}^{t_{m}} \left\la \mathcal A_{h}(I - J_{\dt})\bs e_{h},\mathcal A_{h}\frac{\partial}{\partial t}(J_{\dt}\bs e_{h})\right\ra\dd t \\
 = & 2\la \Theta_{h}^{m} + \bs G^{m},\mathcal A_{h} \bs e_{h}^{m}\ra - 2\int_{0}^{t_{m}} \la\Theta_{h,t} + \bs G_{t},\mathcal A_{h}J_{\dt}\bs e_{h} \ra\dd t \\
&  + 2 \int_{0}^{t_{m}} \left\la \mathcal A_{h}\frac{\partial}{\partial t}(I - J_{\dt})\bs e_{h},\mathcal A_{h}(J_{\dt}\bs e_{h})\right\ra\dd t \\
 \leq & 2\max\limits_{0 \leq i \leq N}\|\mathcal A_{h}\bs e_{h}^{i}\|\left( \max\limits_{0\leq i \leq N}\|\Theta_{h}^{i} + \bs G^{i}\| + \int_{0}^{T}\|\Theta_{h,t} + \bs G_{t}\|\dd t  + \int_{0}^{T}\left\|\mathcal A_{h} \frac{\partial}{\partial t}(I - J_{\dt})\bs e_{h}\right\| \dd t   \right) \\
 \leq & C \max\limits_{0\leq i \leq N}\|\mathcal A_{h}\bs e_{h}^{i}\|  \int_{0}^{T}\left(\|\Theta_{h,t} \| + \|\bs G_{t}\| +\left\|\mathcal A_{h} \frac{\partial}{\partial t}(I - J_{\dt})\bs e_{h}\right\| \right)\dd t \\
 \leq & C\max\limits_{0\leq i \leq N}\|\mathcal A_{h}\bs e_{h}^{i}\| \left[ h^{m} \left( \|\bs u_{t}\|_{L^{1}(H^{m})} + \|\dd^{-} \sigma_{t}\|_{L^{1}(H^{m})} + \|\dd\mu_{t}\|_{L^{1}(H^{m})} + \dt^2 \| \|\mathcal A \bs u_{tt}\|_{L^{\infty}(L^{2})}\right) \right], 
\end{align*}
where in the last inequality, we have used the properties of the one dimensional interpolation operator. Then the desired result follows.
\end{proof}

We summarize the error estimate for the full discretization below.
 \begin{theorem}\label{the:err-full-A}
 Let $\bs u = (\sigma,\mu,\omega)^{\intercal}$ be the solution of \eqref{eq:hodge-operator-form} and $\bs U_{h} = (\tilde\sigma_{h},\tilde\mu_{h},\tilde\omega_{h})^{\intercal} \in \mathcal P_{1}(\bs W_{h})$ be the solution of \eqref{eq:hodge-wave-full-dis}.
Assume that $\sigma_{ttt} \in L^{\infty}((0,T),H^{r}\Lambda^{-})$, $\mu_{ttt} \in L^{\infty}((0,T),H^{r}\Lambda)$ and $\omega_{ttt} \in L^{\infty}((0,T),H^{r}\Lambda^{+})$ with $r \geq 1$. Then, for any $1\leq m \leq r$ we have the bound
\begin{align*}
\max\limits_{0\leq i\leq N} \|\mathcal A\bs u^{i} - \mathcal A_{h}\bs U_{h}^{i}\|  \lesssim & \, h^{m}\left [ \|\bs u_{t}\|_{L^{1}(H^{m})}  + \|\mathcal A \bs u\|_{L^{\infty}(H^{m})} \right.\\
& + \left. \|\dd^{-} \sigma_{t}\|_{L^{1}(H^{m})} + \|\dd\mu_{t}\|_{L^{1}(H^{m})} \right ] \\
 + & \dt^{2} \|\mathcal A \bs u_{tt}\|_{L^{\infty}(L^{2})}.
\end{align*}
\end{theorem}

\section{Numerical experiments}
In this section, we will give some simple numerical examples to illustrate the theoretical results. We consider the Hodge wave equation on the unit square $\Omega = (0,1)^{2}$, i.e., $n = 2$ and compute the cases $k = 0$, $k  = 1$ and $k = 2$.

\subsection{The case $k = 0$} 
The Hodge wave equation presents in the standard $H^{1}(\Omega)$ or $H(\curl)$ language reads as (note that $\sigma = \delta\mu = 0$): Find $\mu \in H_{0}^{1}$ and $\bs \omega \in H_{0}(\rot)$ or $\mu \in H_{0}(\curl)$ and $\bs\omega \in H_{0}(\div)$ such that
\begin{align}
(\mu_{t}, v) + (\bs\omega,\grad v) & = ( f, v)\qquad\forall~~\bs v \in H_{0}^{1}(\Omega) ,\\
(\bs \omega_{t},\bs \phi) - (\grad\mu,\bs \phi) & = 0\qquad\forall~~\bs \phi \in H_{0}(\rot).
\end{align}
or
\begin{align}\label{eq:k=0-1}
(\mu_{t}, v) + (\bs\omega,\curl v) & = ( f, v)\qquad\forall~~\bs v \in H_{0}(\curl) ,\\
\label{eq:k=0-2}
(\bs \omega_{t},\bs \phi) - (\curl\mu,\bs \phi) & = 0\qquad\forall~~\bs \phi \in H_{0}(\div).
\end{align}
We only give the numerical results for \eqref{eq:k=0-1}-\eqref{eq:k=0-2}. We choose the exact solutions as
\begin{align*}
\mu(x,y,t) & = e^{-t} \sin(\pi x)\sin(\pi y), \\
\bs \omega(x,y,t) & = -\pi e^{-t}\begin{pmatrix} 
 \sin(\pi x)\cos(\pi y) \\ -\cos(\pi x)\sin(\pi y)
\end{pmatrix}.
\end{align*}
The initial conditions are $\mu_{0} = \mu(x,y,0)$ and $\bs\omega_{0} = \bs\omega(x,y,0)$. We use piecewise continuous second order polynomial to discrete $\mu$ and use $RT_1$ element \cite{Raviart1977A} to discrete $\omega$, the numerical results are listed in Table \ref{EX-k-0-1}.

\begin{table}[hpt]
\footnotesize
\begin{center}
\caption{Errors and convergence orders in various norms with $\dt = 0.0001$ at $t = 0.0004$.}
\label{EX-k-0-1}
\begin{tabular}{c|c|c|c}
\hline
\hline  $h$ & $\|\mu - \mu_{h}\|$  &$ \|\curl(\mu - \mu_{h})\| $ & $ \|\omega - \omega_{h}\|$ \\
\hline
1/4   &  3.8253e-03  &  1.3417e-01   &  1.3164e-01 \\
\hline
1/8   &  5.0314e-04  &  3.5025e-02   &  3.3567e-02 \\
\hline
1/16  &  6.7850e-05  &  9.5850e-03   &  8.4467e-03 \\
\hline
order     & 2.914 & 1.904 & 1.981   
\\ \hline
\hline
\end{tabular}
\end{center}
\end{table}
From Table \ref{EX-k-0-1}, we can see that the mixed finite element method is of second order convergence rate for the variables $\curl\mu$ and $\omega$, and is of third order convergence rate for $\mu$. 
All these variables have optimal convergence order.

\subsection{The case $k = 1$} 
The Hodge wave equation is: Find $\sigma \in H_{0}(\curl)$, $\bs \mu \in H_{0}(\div)$ and $\omega \in L_{0}^{2}(\Omega)$ such that
\begin{align*}
(\sigma_{t},\tau) - (\bs \mu,\curl\tau) & = 0\qquad\forall~~\tau \in H_{0}(\curl),\\
(\bs \mu_{t},\bs v) + (\curl\sigma,\bs v) + (\omega,\div\bs v) & = (\bs f,\bs v)\qquad\forall~~\bs v \in H_{0}(\div) ,\\
(\omega_{t},\phi) - (\div\bs\mu,\phi) & = 0\qquad\forall~~\phi \in L_{0}^{2}(\Omega).
\end{align*}
This formulation can be viewed as the mixed method for the time-harmonic Maxwell's equations with divergence free constrain on both $\bs \mu$ and $\bs f$. The formulation  can  also be viewed as the mixed method for the elastic wave equation. We use continuous piecewise quadratic polynomial to discrete $\sigma$, use $RT_1$ element \cite{Raviart1977A} to discrete $\mu$ and use discontinuous piecewise linear polynomial to discrete $\omega$.
Firstly, we choose the exact solutions as
\begin{align}\label{ex:sol1}
\sigma(x,y,t) & = 2 e^{-t} \left(\pi\sin^{2}(\pi x)\sin(\pi y)\cos(\pi y)  - x (x-1)(2 x-1) y^{2}(y-1)^{2} \right)\\
\bs \mu(x,y,t) & = e^{-t} \begin{pmatrix} \sin^{2}(\pi x)\sin^{2}(\pi y)  \\
x^{2}(x-1)^{2}y^{2}(y-1)^{2}\end{pmatrix},\\
\label{ex:sol3}
\begin{split}
\omega(x,y,t) & = -2 e^{-t}\left(\pi\sin(\pi x)\cos(\pi x)\sin^{2}(\pi y)\right. \\
&\quad \left. + x^{2}(x-1)^{2}y(y-1)(2y-1) \right),
\end{split}
\end{align}
with initial conditions
$\sigma_{0} = \sigma(x,y,0)$, $\bs\mu_{0} = \bs \mu(x,y,0)$ and $\omega_{0} = \omega(x,y,0)$. The numerical results are listed in Table \ref{EX1-1}. We also test the long time robustness of our algorithm, the numerical results are listed in Table \ref{EX1-3}. 

\begin{table}[h!]
\footnotesize
\begin{center}
\caption{Errors and convergence orders in various norms with $\dt = 0.0001$ at $t = 0.0004$.}
\label{EX1-1}
\begin{tabular}{c|c|c|c|c|c}
\hline
\hline  $h$ & $\|\sigma - \sigma_{h}\|$  &$ \|\curl(\sigma - \sigma_{h})\| $ &$ \|\mu - \mu_{h}\|$ & $\|\div(\mu - \mu_{h})\|$ & $\|\omega - \omega_{h}\|$\\
\hline
1/4   &  4.8563e-02  &  1.6691e+00   &  2.9085e-02    &    0.1391   &  0.1388 \\ 
\hline
1/8   &  6.4968e-03  &  4.4475e-01   &  7.6216e-03    &    0.0364   &  0.0367 \\ 
\hline
1/16  &  8.2584e-04  &  1.1298e-01   &  1.9354e-03    &    0.0093   &  0.0093 \\ 
\hline
order    &  2.949     &  1.942     &    1.950   &  1.950  &   1.9763 \\  
\hline
\hline
\end{tabular}
\end{center}
\end{table}

\begin{table}[h!]
\footnotesize
\begin{center}
\caption{Long time problem with $\dt = 0.1$ and $h = 1/16$.}
\label{EX1-3}
\begin{tabular}{c|c|c|c|c|c}
\hline
\hline  $T$ & $\|\sigma - \sigma_{h}\|$  &$ \|\curl(\sigma - \sigma_{h})\| $ &$ \|\mu - \mu_{h}\|$ & $\|\div(\mu - \mu_{h})\|$ & $\|\omega - \omega_{h}\|$\\
\hline
10   &  5.4138e-04  &  1.5087e-02   &  3.7500e-01    &    1.3603   &  0.2374 \\ 
\hline
30   &  4.8684e-04  &  1.8186e-02   &  3.7502e-01    &    1.3604   &  0.2486 \\ 
\hline
50  &  6.1267e-04  &  1.4339e-02   &  3.7502e-01    &    1.3604   &  0.4158 \\ 
\hline
\hline
\end{tabular}
\end{center}
\end{table}


Then, we chose $\bs f =0$ and the initial conditions $\sigma_{0} = \sigma(x,y,0)$, $\bs\mu_{0} = \bs \mu(x,y,0)$ and $\omega_{0} = \omega(x,y,0)$ with $\sigma$, $\bs\mu$ and $\omega$ defined as in \eqref{ex:sol1} - \eqref{ex:sol3}. We compute the energies $\|\bs U_{h}^{i}\|$ and $\|\mathcal A_{h}\bs U_{h}^{i}\|$ on different time levels, the numerical results are showing in Fig. \ref{Fig:EX2-1}.


\begin{figure}
\centering
\includegraphics[width=0.80\textwidth]{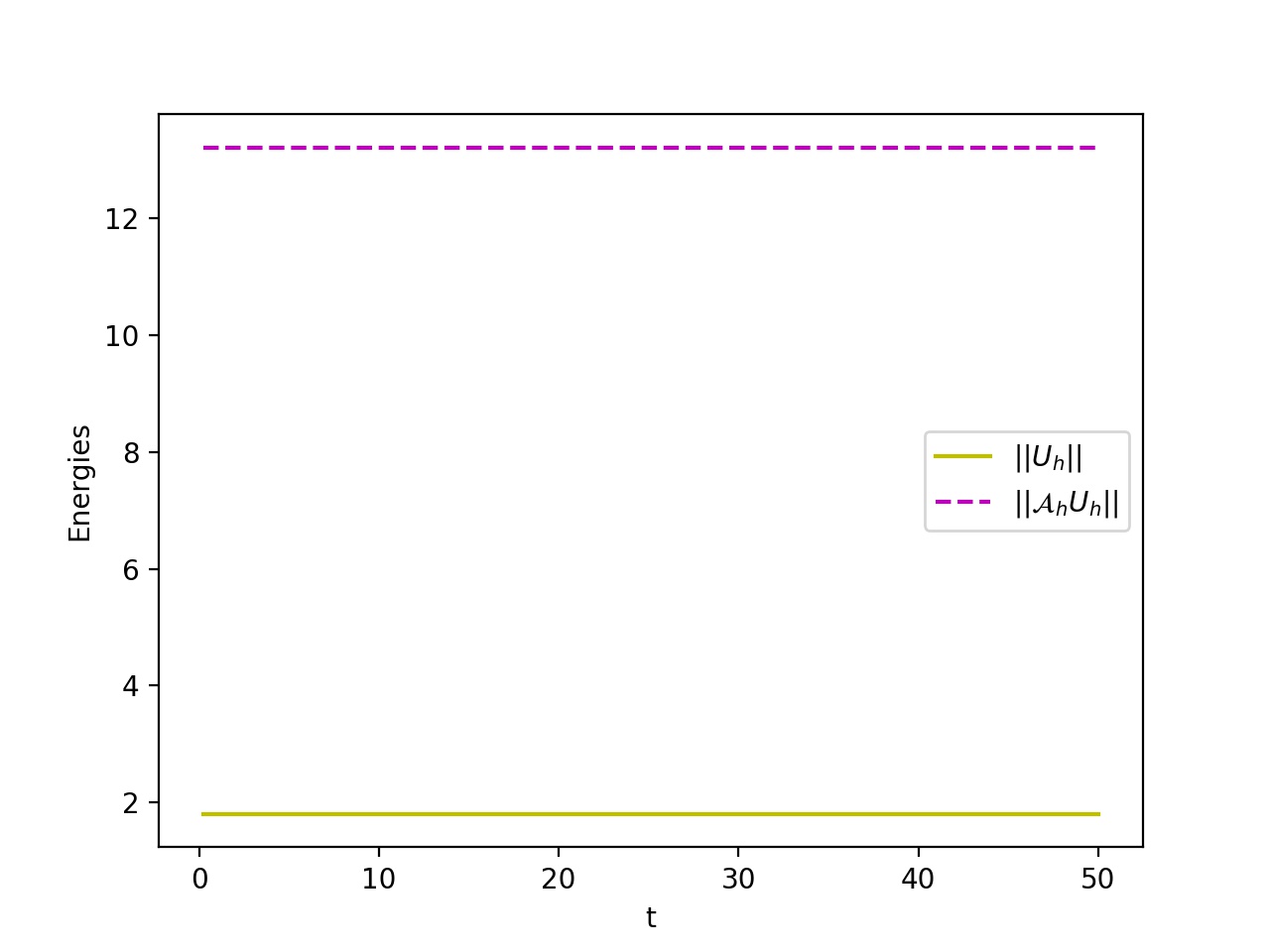}
\caption{Energies $\|\bs U_h\|$ and $\|\mathcal A_h\bs U_h\|$ in different times with $h = 1/16$ and $\dt = 0.25$.}\label{Fig:EX2-1}
\end{figure}

From this example, we have the following observations.
\begin{enumerate}
\item The mix finite element method is of second order convergence rate for the variables $\curl\sigma$, $\mu$, $\div\mu$ and $\omega$, and is of third order convergence rate for the variable $\sigma$. All of these variables have optimal convergence order.
\item From Table \ref{EX1-3}, we can see that the mixed finite element method is robust for long-time problem.
\item From Fig. \ref{Fig:EX2-1}, we can see that the mixed finite element method conserves the energies $\|\bs U_h\|$ and $\|\mathcal A_h\bs U_h\|$ exactly.
\end{enumerate}

\subsection{The case $k = 2$} The Hodge wave equation presents in the $H(\div)$ and $L^{2}$ language reads as (note that $\omega = \dd\mu = 0$): Find $\bs \sigma \in H_{0}(\div)$ and $\mu \in L_{0}^{2}(\Omega)$ such that 
\begin{align}
(\bs \sigma_{t},\bs \tau) + (\mu,\div\bs\tau)  & = 0\qquad\forall~~\bs\tau \in H_{0}(\div), \\
(\mu_{t}, v) - (\div\bs\sigma, v) & = (f, v) \qquad\forall ~~ v \in L_{0}^{2}(\Omega).
\end{align}
This formulation is the mixed method for acoustic wave equations \cite{Kirby_2014}. We choose the exact solutions as
\begin{align}
\bs\sigma(x,y,t)  & =-\pi e^{-t} \begin{pmatrix} \cos(\pi x)\sin(\pi y) \\ \sin(\pi x)\cos(\pi y) \end{pmatrix},\\
\mu(x,y,t) & = e^{-t}\sin(\pi x) \sin(\pi y),
\end{align}
and pick initial conditions $\bs\sigma_{0} = \bs\sigma(x,y,0)$ and $\mu_{0} = \mu(x,y,0)$. We use $RT_1$ element to discrete $\bs\sigma$ and use discontinuous piecewise linear polynomials to discrete $\mu$, the numerical results are listed in Table \ref{EX-k-2-1}.

\begin{table}[h!]
\footnotesize
\begin{center}
\caption{Errors and convergence orders in various norms with $\dt = 0.0001$ at $t = 0.0004$.}
\label{EX-k-2-1}
\begin{tabular}{c|c|c|c}
\hline
\hline  $h$ & $\|\sigma - \sigma_{h}\|$  &$ \|\div(\sigma - \sigma_{h})\| $ &$ \|\mu - \mu_{h}\|$ \\
\hline
1/4   &  5.6058e-02  &  3.8201e-01   &  1.9350e-02 \\
\hline
1/8   &  1.4002e-02  &  9.6901e-02   &  4.9051e-03 \\
\hline
1/16  &  3.4958e-03  &  2.4360e-02   &  1.2312e-03  \\
\hline
order         & 2.002  & 1.985 & 1.992  
\\ \hline
\hline
\end{tabular}
\end{center}
\end{table}

From Table \ref{EX-k-2-1}, we can see that the mixed finite element method is of second order convergence rate for all the variables.

\section*{Acknowledgments}
We would like to thank Professor Long Chen from University of California at Irvine for valuable discussion and suggestions.


\end{document}